\documentclass[reqno,11pt,twoside]{article}
\usepackage[hmargin=1in, vmargin=1.in, a4paper, centering]{geometry}
\usepackage[T1]{fontenc}
\usepackage[lining]{ebgaramond}
\usepackage{amsmath,amsthm}
\usepackage[ebgaramond]{newtxmath}
\usepackage{accents}
\usepackage{enumerate}
\usepackage[font=small, labelfont=bf, labelsep=colon, margin=0.5in]{caption}
\usepackage{graphicx}

\usepackage{xcolor}
\definecolor{darkgreen}{rgb}{0.00, 0.67, 0.00}
\usepackage{bm}
\usepackage[toc,titletoc]{appendix}
\appendixtocoff
\usepackage{etoolbox, apptools}
\AtBeginEnvironment{appendices}{\appendixtrue}
\usepackage{longtable,booktabs,array}
\usepackage{algorithm}
\usepackage{algpseudocode}
\usepackage{marvosym}
\usepackage[section]{placeins}
\numberwithin{equation}{section}
\ifdefined\iflatexml
\theoremstyle{plain}
\newtheorem{theorem}{Theorem}[section]
\newtheorem{prop}[theorem]{Proposition}
\newtheorem{lemma}[theorem]{Lemma}

\theoremstyle{definition}
\newtheorem{definition}[theorem]{Definition}
\newtheorem{remark}[theorem]{Remark}
\else
\usepackage{tcolorbox,varwidth}
\tcbuselibrary{skins, breakable, theorems}
\colorlet{thmcolor}{orange!80!white}
\colorlet{remcolor}{teal!50!white}
\colorlet{defncolor}{blue!50!black}
\usepackage{keytheorems}
\tcbset{commonstyle/.style={%
colback=white!95!tcbcolframe,
colbacktitle=white!80!tcbcolframe,
arc=2mm,
fonttitle=\bfseries,
coltitle=black,
enhanced,
varwidth boxed title*=-2cm,
attach boxed title to top left={yshift=-3mm,xshift=0.5cm,yshifttext=-1mm},
description font=\mdseries,
breakable,
before upper={\parindent15pt\noindent},
beforeafter skip balanced=10.0pt plus 1.0pt minus 1.0pt,
separator sign={\textbf{: }},
overlay first={\draw[color=tcbcolframe, line width=.5pt] (frame.south west)--(frame.south east);},
overlay middle={\draw[color=tcbcolframe, line width=.5pt] (frame.south west)--(frame.south east);\draw[color=tcbcolframe, line width=.5pt] (frame.north west)--(frame.north east);},
overlay last={\draw[color=tcbcolframe, line width=.5pt] (frame.north west)--(frame.north east);}
}
}
\newkeytheoremstyle{mythmstyle}{
bodyfont=\normalfont,
headpunct={},
notebraces={}{},
noteseparator={: }, 
notefont=\bfseries,
tcolorbox = {colframe=thmcolor,commonstyle}
}
\newkeytheoremstyle{myremstyle}{
bodyfont=\normalfont,
headpunct={},
notebraces={}{},
noteseparator={: }, 
notefont=\bfseries,
tcolorbox = {colframe=remcolor,commonstyle}
}
\newkeytheoremstyle{mydefstyle}{
bodyfont=\normalfont,
headpunct={},
notebraces={}{},
noteseparator={: }, 
notefont=\bfseries,
tcolorbox = {colframe=defncolor,commonstyle}
}
\newkeytheorem{theorem}[
  name=Theorem,
  style=mythmstyle,
  parent=section
]
\newkeytheorem{lemma}[
  name=Lemma,
  style=mythmstyle,
 sibling=theorem
]
\newkeytheorem{prop}[
  name=Proposition,
  style=mythmstyle,
 sibling=theorem
]
\newkeytheorem{corollary}[
  name=Corollary,
  style=mythmstyle,
 sibling=theorem
]
\newkeytheorem{remark}[
  name=Remark,
   sibling=theorem,
  style=myremstyle
]
\newkeytheorem{definition}[
  name=Definition,
  style=mydefstyle,
 sibling=theorem
]
\fi
\usepackage{mleftright} 
\mleftright 

\newcommand{\N}{\mathcal{N}}

\newcommand{\Dir}{{\mathrm{D}}}
\newcommand{\Neu}{{\mathrm{N}}}
\newcommand{\dr}{\mathrm{d}}

\newcommand{\er}{\mathrm{e}}

\renewcommand{\epsilon}{\varepsilon}

\newcommand{\Pt}{\mathcal{P}}
\newcommand{\Lfloor}{\left\lfloor}
\newcommand{\Rfloor}{\right\rfloor}
\newcommand{\entire}[1]{\Lfloor #1 \Rfloor}
\newcommand{\ceiling}[1]{\left\lceil#1\right\rceil}

\newcommand{\ball}[1]{\mathbb{B}^{#1}}
\renewcommand{\tilde}{\widetilde}
\DeclareMathAccent{\wtilde}{\mathord}{largesymbols}{"65}
\newcommand{\utilde}[1]{\underaccent{\wtilde}{#1}}

\renewcommand\footnotemark{}
\usepackage[nobottomtitles,pagestyles]{titlesec}
\titleformat{\section}
{\normalfont\large\bfseries}
{\filcenter\IfAppendix{Appendix }{\S}\thesection.}{1ex}{\filcenter}
\titleformat{\subsection}
{\normalfont\bfseries}
{\filcenter\S\thesubsection.}{1ex}{\filcenter}
\usepackage[colorlinks,allcolors=blue,pagebackref]{hyperref}
\usepackage{pifont}
\renewcommand*{\backrefalt}[4]{%
\ifcase #1 %
\textcolor{red}{No citations}%
\or
\ding{43}~p.~#2%
\else
\ding{43}~pp.~#2%
\fi}
\usepackage{breakcites}
\newcommand{\mydoi}[1]{\href{https://doi.org/#1}{doi: #1}}
\newcommand{\myarXiv}[1]{\href{https://arxiv.org/abs/#1}{arXiv: #1}}
\begin{document}
\title{%
P\'{o}lya's conjecture for higher-dimensional Neumann balls%
}
\author{
Nikolay Filonov
\thanks{%
\textbf{N. F.: }St.~Petersburg Department
of Steklov Institute of Mathematics of RAS,
Fontanka 27, 191023, St. Petersburg, Russia;
\href{mailto:filonov@pdmi.ras.ru}{\nolinkurl{filonov@pdmi.ras.ru}}; ORCID: 0000-0002-1586-3031%
}
\and
Michael Levitin\hspace{-3ex}
\thanks{%
\textbf{M. L.: }Department of Mathematics and Statistics, University of Reading, 
Pepper Lane, Whiteknights, Reading RG6 6AX, UK;
\href{mailto:M.Levitin@reading.ac.uk}{\nolinkurl{M.Levitin@reading.ac.uk}}; \url{https://www.michaellevitin.net}; ORCID: 0000-0003-0020-3265%
}
\and 
Iosif Polterovich
\thanks{%
\textbf{I. P.: }D\'e\-par\-te\-ment de math\'ematiques et de statistique, Univer\-sit\'e de Mont\-r\'eal, 
CP 6128 succ Centre-Ville, Mont\-r\'eal QC  H3C 3J7, Canada;
\href{mailto:iosif.polterovich@umontreal.ca}{\nolinkurl{iosif.polterovich@umontreal.ca}}; \url{https://www.dms.umontreal.ca/\~iossif}; ORCID: 0009-0007-0052-6589%
}
\and
David A. Sher
\thanks{%
\textbf{D. A. S.:  }Department of Mathematical Sciences, DePaul University, 2320 N. Kenmore Ave, 60614, Chicago, IL, USA;
\href{mailto:dsher@depaul.edu}{\nolinkurl{dsher@depaul.edu}}; ORCID: 0009-0003-2478-1083%
}
}
\date{\small arXiv v.2; 21 September 2026} 
\maketitle
\begin{abstract} 
We prove P\'olya's conjecture for the Neumann eigenvalues of the Laplacian on Euclidean balls in dimensions three and higher. The proof further develops the approach introduced in our earlier work on the two-dimensional case and on Dirichlet eigenvalues in arbitrary dimensions. The main difficulty in the higher dimensional Neumann case is that one has to estimate zeros of the derivatives of ultraspherical Bessel functions, rather than of the usual Bessel functions.
For low-lying eigenvalues,  we use variational estimates involving dimension-dependent test functions, which is a novel  ingredient allowing  us to control a larger dimension-scaled frequency range. Other components of the proof include phase-function bounds, lattice-point counting techniques, and computer-assisted arguments. 
\end{abstract}


{\small \tableofcontents}
\section{Introduction}\label{sec:introduction}

Let $\Omega\subset\mathbb{R}^d$ be a bounded Lipschitz domain. Consider the Neumann eigenvalue problem 
\begin{equation}\label{eq:neuL}
\begin{split}    
-\Delta u=\mu u \qquad&\text{in }\Omega,\\ 
\partial_n u=0\qquad&\text{on }\partial\Omega,
\end{split}
\end{equation}
where $\partial_n u = \left.\langle \nabla u, n\rangle\right|_{\partial\Omega}$ denotes the exterior normal derivative of $u$.
The spectrum of \eqref{eq:neuL} is
\[
0=\mu_1(\Omega)<\mu_2(\Omega)\le\dots\le \mu_n(\Omega)\le \dots,
\]
enumerated with multiplicity. Denote the Neumann counting function by
\[
\N^\Neu_\Omega(\lambda):=\#\left\{n: \mu_n(\Omega)\le \lambda^2\right\}.
\]
Weyl's law for the Neumann Laplacian states that
\[
 \N^\Neu_\Omega(\lambda)
 = C_d |\Omega|\lambda^d+o(\lambda^d),
 \qquad \lambda\to+\infty,
\]
where
\[
 C_d:=
 \frac{1}
 {(4\pi)^{d/2}\Gamma\left(\frac d2+1\right)}
\]
is Weyl's constant and $|\Omega|$ denotes the $d$-dimensional
volume of $\Omega$, see, for example, \cite{Wey11,SV97}.
Thus, the principal term in the asymptotic distribution of the
Neumann eigenvalues depends only on the volume of the domain.

For sufficiently regular domains, Weyl's two-term conjecture \cite{Wey12} predicts
the sharper asymptotic formula
\begin{equation}\label{eq:Weyl-two-term}
\N^\Neu_\Omega(\lambda) = C_d |\Omega|\lambda^d + C_{b,d}|\partial\Omega|\lambda^{d-1} + o(\lambda^{d-1}),
\qquad \lambda\to+\infty,
\end{equation}
where $|\partial\Omega|$ denotes the $(d-1)$-dimensional measure of
the boundary, and $C_{b,d}>0$ is a positive constant (see, e.g. \cite{FLPS}).   Consequently, whenever \eqref{eq:Weyl-two-term}
holds, one has
\begin{equation}
\label{eq:Polyagen}
\N^\Neu_\Omega(\lambda)> C_d|\Omega|\lambda^d
\end{equation}
for all sufficiently large $\lambda$. The two-term asymptotic formula
remains conjectural in full generality, although it is known under
the standard non-periodicity condition on the billiard flow
\cite{Ivr80}. In particular, it holds for real analytic convex domains. 
We refer to \cite[\S1]{FLPS} for a more
detailed discussion.

P\'olya's conjecture for the Neumann Laplacian asserts that inequality 
\eqref{eq:Polyagen} is valid not only
asymptotically, but at {\em any}  value $\lambda \ge 0$.
The main content of the conjecture is therefore a uniform estimate
for the entire spectrum, including the low- and intermediate-frequency
ranges which are not controlled by Weyl's law. 

Note that P\'olya's conjecture for the Neumann Laplacian can be equivalently restated in terms of individual eigenvalues rather than the counting function:
\begin{equation}\label{eq:Polyan}
\left(\mu_n(\Omega)\right)^d\le \left(\frac{n-1}{C_d|\Omega|}\right)^2\qquad \text{for all } n=2,3,\dots.
\end{equation}

There is a parallel P\'olya conjecture for the Dirichlet Laplacian. If
\[
 0<\lambda_1(\Omega)\leq\lambda_2(\Omega)\leq\cdots
\]
are the Dirichlet eigenvalues, enumerated with multiplicity, and
\[
 \N^\Dir_\Omega(\lambda)
 :=
 \#\bigl\{n:\lambda_n(\Omega)\leq\lambda^2\bigr\},
\]
then the Dirichlet P\'olya conjecture states that
\[
 \N^\Dir_\Omega(\lambda)
 \leq
 C_d|\Omega|\lambda^d
 \qquad\text{for every }\lambda\geq0.
\] 
Thus, the Dirichlet and Neumann counting functions are conjectured
to lie, respectively, below and above the principal term in Weyl's
law.

P\'olya proved the conjecture for domains which tile Euclidean space
\cite{Pol54,Pol61}, that is, domains whose congruent copies cover
$\mathbb R^d$, up to a set of measure zero, without overlapping.
Additional restrictions occurring in the original Neumann argument
were subsequently removed, see \cite{Kel66}. For arbitrary
domains, the full conjecture remains open. The weaker universal
bounds
\[
 \left(\frac{d}{d+2}\right)^{d/2}
\mathcal{N}^\Dir_\Omega(\lambda)
 \leq
 C_d|\Omega|\lambda^d
 \leq
 \frac{d+2}{2} \mathcal{N}^\Neu_\Omega(\lambda)
\]
follow from the Berezin--Li--Yau and Kr\"oger inequalities, see
\cite{LiYau, Kro, Lap}.

Beyond tiling domains, comparatively few cases of P\'olya's
conjecture are known. In our earlier work \cite{FLPS}, we proved
both the Dirichlet and Neumann conjectures for the planar disk and
for arbitrary planar circular sectors. In the same paper, we proved
the Dirichlet conjecture for Euclidean balls in every dimension.
The Dirichlet conjecture has also recently been established for
annuli \cite{FLPS-Ann}.  Some other recent advances can be found  in \cite{FrSa23, FLPS-AB, GJWY24, HeWa24, GMWZ25, JiLi25, GaJiLi26}.
The higher-dimensional Neumann problem for
balls, which was left open in \cite{FLPS}, is the subject of the
present paper.

Let $\mathbb{B}^d$ denote the unit ball in $\mathbb R^d$, where $d\geq3$.
Since
\[
\left |\mathbb{B}^d\right|
 =
 \frac{\pi^{d/2}}
 {\Gamma\left(\frac d2+1\right)},
\]
the principal Weyl term for $\mathbb{B}^d$ is
 \[
 C_d\left|\mathbb{B}^d\right|\lambda^d
 =
 w_d\lambda^d,
 \qquad
 w_d:=
 \frac{1}
 {2^d\left(\Gamma\left(\frac d2+1\right)\right)^2}.
 \]
Our main result is that  P\'olya's conjecture for the Neumann eigenvalues of balls is valid in every dimension $d\ge 3$. Together with the two-dimensional result of \cite{FLPS}, this proves the conjecture for Euclidean balls in all dimensions $d\ge 2$.
\begin{theorem}\label{thm:main}
Let $d\geq3$. Then
\[
 \N^\Neu_{B^d}(\lambda)
 \geq
 w_d\lambda^d
 \qquad\text{for every }\lambda\geq0.
\]
\end{theorem}

\begin{remark}\label{rem:concurrent}
When the present paper was in the final stages of preparation, an independent proof of 
Theorem \ref{thm:main} appeared in \cite{Li}. While both papers use a similar collection of ideas, their technical implementations are different. 
\end{remark}

The proof  of Theorem \ref{thm:main} further develops the approach introduced in \cite{FLPS}.
The Neumann eigenvalues of a ball are described by the zeros of
derivatives of ultraspherical Bessel functions, with
dimension-dependent harmonic multiplicities. This makes the
higher-dimensional Neumann case substantially more delicate than
both the planar Neumann problem and the higher-dimensional
Dirichlet problem.

For large values of the spectral parameter, we compare these zeros
with zeros of derivatives of the usual Bessel functions and reduce
the problem to estimates for weighted lattice-point sums. The
resulting bounds combine estimates for Bessel phase functions with
combinatorial estimates for the harmonic multiplicities. At low
frequencies, we use the variational characterisation of the first
radial eigenvalue in each angular momentum sector.  This is a particularly novel feature of the argument. 
The remaining finite ranges are handled by a rigorous computer-assisted argument using exact
rational arithmetic and certified enclosures for the relevant
zeros.

\subsection*{Plan of the paper}
\addcontentsline{toc}{subsection}{Plan of the paper} 
The paper is organised as follows. In \S\ref{sec:prelim} we introduce the notation, recall the required facts about ultraspherical Bessel functions and the zeros of their derivatives, and summarise the frequency ranges covered by the different parts of the proof. The next section treats large values of $\lambda$: comparison with the zeros of derivatives of the usual Bessel functions reduces the problem to a weighted lattice-point estimate, which is then established using phase-function bounds and estimates for the harmonic multiplicities. This is followed by the low-frequency analysis in \S \ref{sec:lowlambda}, where we first deal with very small values of $\lambda$ and then develop variational estimates based on constant and dimension-dependent radial test functions. In \S\ref{sec:gap} we complete the proof by a certified computer-assisted verification of the remaining finite ranges. Properties of the function $G_\lambda$ defined by \eqref{eq:Glambda} and the combinatorial estimates used in the argument are collected in Appendices \ref{app:G} and \ref{app:comb}, respectively, while the proofs of several technical statements are deferred to Appendix \ref{app:technical}. Appendix \ref{app:rational} contains details of implementation of a computer-assisted algorithm.

\subsection*{Acknowledgements}  
\addcontentsline{toc}{subsection}{Acknowledgements} 
Research of I.P.\ was partially supported by NSERC and FRQNT.  D.S.\ is grateful for support from the AMS-Simons research enhancement grant 501949-9208.

The authors would like to thank the Isaac Newton Institute for Mathematical Sciences, Cambridge, where some work on this paper was undertaken, for support and hospitality during the programme \emph{Geometric spectral theory and applications}, supported by EPSRC grant EP/Z000580/1. M.L.'s stay at INI was supported by a Simons fellowship.

\subsection*{Data availability statement}  
\addcontentsline{toc}{subsection}{Data availability statement} 

The accompanying \texttt{Mathematica} script and its printout are available for download at 
\begin{center}
\url{https://www.michaellevitin.net/polya.html#neumann}\quad and\quad \url{https://github.com/michaellevitin/Polya}.
\end{center}
The main script reproduces all the figures and provides certified rational enclosures for the majority of floating-point constants appearing in the paper. The additional script, also available in Python, implements the algorithm described in \S\ref{sec:gap} and Appendix \ref{app:rational}.

\subsection*{AI usage disclosure} 
\addcontentsline{toc}{subsection}{AI usage disclosure} 
The authors acknowledge the use of several models of ChatGPT and Claude for mathematical discussions and editorial assistance. In particular, AI tools contributed to the development of several technical lemmas in Appendix \ref{app:technical}, as well as to the design and implementation of the rigorous computer-assisted algorithm in \S\ref{sec:gap}. All AI-assisted arguments and computations were independently checked by the authors, who take full responsibility for the contents of the paper.

\section{Preliminaries}\label{sec:prelim}

\subsection{Notation and setup}

Throughout the paper, $J_\nu(x)$ denotes a Bessel function of order $\nu\ge 0$, $j_{\nu, k}$, $k\in\mathbb{N}$, is the $k$th positive zero of $J_\nu$, and $j'_{\nu, k}$ is the $k$th positive zero of $\frac{\dr}{\dr x}J_\nu(x)$, with the exception of $j'_{0,1}:=0$, see \cite{Wat, PalApa}. We will additionally denote $j_{\nu, 0}:=0$.

For $d\ge 3$ and $m\in \mathbb{N}_0:=\{0\}\cup\mathbb{N}$, let
\[
P_{d,m}(x):=x^{-\left(\frac{d}{2}-1\right)} J_{m+\frac{d}{2}-1}(x)
\]
denote a \emph{$d$-dimensional ultraspherical Bessel function of order $m$}. We denote by $p'_{d,m,k}$ the $k$th positive zero of its derivative 
\[
 \frac{\dr}{\dr x} P_{d,m}(x) = x^{-\frac{d}{2}}\left(x J'_{m+\frac{d}{2}-1}(x)-\left(\frac{d}{2}-1\right)J_{m+\frac{d}{2}-1}(x)\right) = x^{-\frac{d}{2}} \left(m J_{m+\frac{d}{2}-1}(x)-x J_{m+\frac{d}{2}}(x)\right),
\]
with the exception of $p'_{d,0,1}:=0$.

The Neumann eigenvalues of the $d$-dimensional ball are
\[
\mu_{d,m,k}:=\left(p'_{d,m,k}\right)^2, \qquad m\in\mathbb{N}_0, k\in\mathbb{N}.
\]
Each eigenvalue $\mu_{d,m,k}$ has a multiplicity
\begin{equation}\label{eq:kappadm}
\kappa_{d,m} =\binom{m+d-1}{d-1}-\binom{m+d-3}{d-1},
\end{equation}
which is the dimension of the space of homogeneous harmonic polynomials of degree $m$ in $\mathbb{R}^d$, or of the space
\[
\mathcal{H}_{d,m}:=\left\{\chi\in C^\infty\left(\mathbb{S}^{d-1}\right): -\Delta_{\mathbb{S}^{d-1}} \chi= m(m+d-2) \chi\right\}
\]
of spherical harmonics of degree $m$ on $\mathbb{S}^{d-1}$.

Thus,
\begin{equation}\label{eq:defofneumcount}
\mathcal N^\Neu_{\mathbb B^d}(\lambda) = \sum_{m=0}^{\infty}\kappa_{d,m}\#\left\{k\in\mathbb N\ :\ p'_{d,m,k}\le\lambda\right\}.
\end{equation}

\subsection{Properties of ultraspherical Bessel functions and zeros of their derivatives}

We start with

\begin{lemma}\label{lem:trans}
Let $d\ge 3$, $m\in\mathbb{N}_0$, and $k\in\mathbb{N}$. Every zero $p'_{d,m,k}$ is simple, and, except for $p'_{d,0,1}=0$, a transcendental number. 
\end{lemma}

The proof can be found in \cite{LoMu95}: it extends the classical result of Siegel for the zeros of Bessel functions.

The following are the standard bounds for the zeros of derivatives of ultraspherical Bessel functions, see \cite{Spigler}.

\begin{lemma}\label{lem:pandj} 
Let $d\ge 3$, $m\in\mathbb{N}_0$ and $k\in\mathbb{N}$. 
\begin{itemize}
\item If $m=0$, then $p'_{d,0,k}=j_{\frac{d}{2},k-1}$.
\item If $m>0$, then
\[
p'_{d,m,k}\in \left(j_{m+\frac{d}{2}-1,k-1}, j_{m+\frac{d}{2}-1,k}\right)
\]
with $\operatorname{sign} P'_{d,m}\left(j_{m+\frac{d}{2}-1,k}\right)=(-1)^k$. 
\item
\begin{equation}\label{eq:uvsj}
p'_{d,m,k}< j'_{m+\frac{d}{2}-1,k}\qquad\text{for all }k\in\mathbb{N}.
\end{equation}
\end{itemize}
\end{lemma}

For additional properties of ultraspherical Bessel functions and zeros of their derivatives we refer to \cite{IS88, LorSze94, L99, FLPSbessel, Jiang}. In particular, we mention the lower bound \cite{LorSze94},
\begin{equation}\label{eq:LS}
p'_{d,m,1}>\sqrt{m(m+d-2)}, \qquad m\ge 1.
\end{equation}

The standard variational principle for the Neumann Laplacian, see e.g. \cite{LMP},
\[
\mu_k(\Omega) = \min_{\substack{\mathcal{S}\subset H^1(\Omega)\\\dim\mathcal{S}=k}} \max_{u\in\mathcal{S}\setminus\{0\}} \frac{\|\nabla u\|^2_{L^2(\Omega)}}{\|u\|^2_{L^2(\Omega)}},
\]
applied to $\Omega=\mathbb{B}^d$ and restricted to test functions $u=\rho\chi$ with $\rho\in H^1\left((0,1), x^{d-1}\,\dr x\right)$, $\chi\in\mathcal{H}_{d,m}$,  gives, for any such function $\rho$ (which may depend on $d$ and/or $m$), 
\begin{equation}\label{eq:varpr}
\mu_{d,m,1}=(p'_{d,m,1})^2 \le A_d(m):= \frac{\int_0^1\left(|\rho'(x)|^2 x^{d-1}+m(m+d-2)|\rho(x)|^2x^{d-3}\right)\, \dr x}{\int_0^1|\rho(x)|^2x^{d-1}\, \dr x}.
\end{equation}
We will use this variational principle extensively, with different test functions.

\subsection{An illustration of  main results}
As mentioned in the Introduction, our proof of Theorem \ref{thm:main} splits into three main parts. It will be often convenient to introduce a new parameter $\tau$ via
\begin{equation}\label{eq:lambdatau}
\lambda = \tau^3 d^{3/2}.
\end{equation}
We therefore intend to prove that P\'olya's inequality
\begin{equation}\label{eq:Polyatau}
\mathcal{N}^\Neu_{\ball{d}}\left(\tau^3 d^{3/2}\right)-w_d \tau^{3d} d^{3d/2}> 0
\end{equation}
holds for all $d\ge 3$ and all $\tau>0$.

Set 
\[
\tau^*:=\frac{28}{25}.
\]
In \S\ref{sec:largelambdas}, we show (Theorem \ref{thm:A}) the existence of a non-increasing sequence $\{\tau_d^*\}$ with $\tau_d^*\ge 1$ for all $d$,  $\tau_d^*=1$ for $d\ge 61$, and $\tau_d^*<\tau^*$ for $d\ge 13$, 
such that \eqref{eq:Polyatau} holds for all $\tau>\tau_d^*$, see Figure \ref{fig:totalpictsmall}.

In \S\ref{sec:lowlambda}, we show (Theorem \ref{thm:low}) that for all $d\ge 3$,  \eqref{eq:Polyatau} holds for all $0\le \tau\le \tau^*$.

\begin{figure}[htb]
\centering
\includegraphics{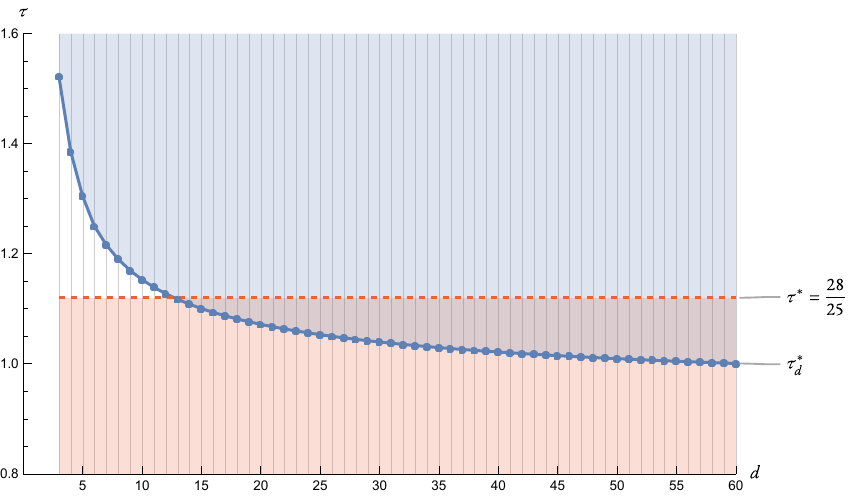}
\caption{Summary of analytic results in $(d,\tau)$-range.\label{fig:totalpictsmall}}
\end{figure}

The combination of these two results leaves a gap,  for $3\le d \le 12$, namely $\lambda\in\left({\tau^*}^3 d^{3/2}, {\tau_d^*}^3 d^{3/2}\right]$, see Figure \ref{fig:lambdagap}. We fill this gap in \S\ref{sec:gap} by using a rigorous computer-assisted algorithm to show that  inequality \eqref{eq:Polyan} holds, in these dimensions, for all $n$ such that $\mu_n\left(\mathbb{B}^d\right)\le {\Lambda_{d}^*}^2$, where $\Lambda_{d}^*:=\ceiling{{\tau_d^*}^3 d^{3/2}}$.

\begin{figure}[htb]
\centering
\includegraphics{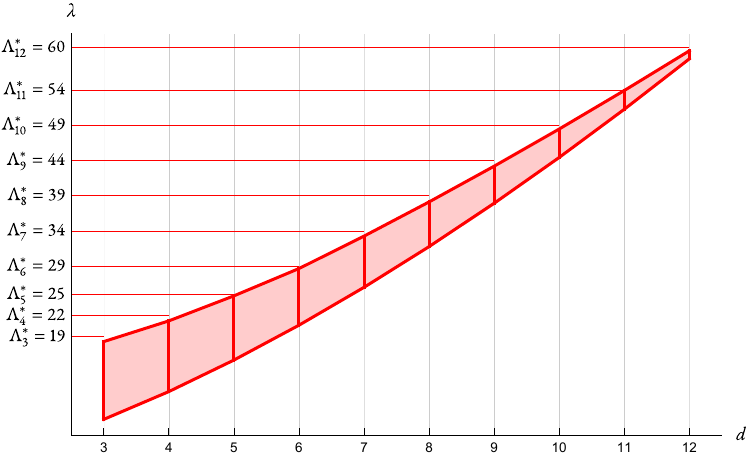}
\caption{Remaining $\lambda$ gaps.\label{fig:lambdagap}}
\end{figure}

\section{Large $\lambda$ via comparison with zeros of Bessel derivatives}\label{sec:largelambdas}

\subsection{Main result and approach}

As a consequence of \eqref{eq:uvsj} and \eqref{eq:defofneumcount},
\begin{equation}\label{eq:modneumcount}
\mathcal N^\Neu_{\mathbb B^d}(\lambda)\ge \sum_{m=0}^{\infty}\kappa_{d,m}\#\left\{k\in\mathbb N: j'_{m+\frac{d}{2}-1, k}\le\lambda\right\} =: \utilde{\mathcal{N}}^\Neu_{\mathbb B^d}(\lambda)
\end{equation}
This allows us to relate the eigenvalue counting functions to a lattice point count. Specifically, due to \cite[Theorem 1.10]{FLPSbessel}, see also \cite[Proposition 3.1]{FLPS}, we have
\begin{equation}\label{eq:latticecomparison}
\utilde{\mathcal N}^\Neu_{\mathbb B^d}(\lambda)\ge \sum_{m=0}^{\entire{\lambda-\frac{d}{2}+1}}\kappa_{d,m}\entire{G_\lambda\left(m+\frac{d}{2}-1\right)+\frac{3}{4}} =: \Pt_d^\Neu(\lambda),
\end{equation}
where 
\begin{equation}\label{eq:Glambda}
G_\lambda(z):=\frac{1}{\pi}\left(\sqrt{\lambda^2-z^2}-z\arccos\frac{z}{\lambda}\right),\qquad z\in[0,\lambda],
\end{equation} 
extended by zero for $z>\lambda$, is a function originally introduced in \cite{kufe}; see Appendix \ref{app:G} for its various properties.

In this section we show analytically that for large $\lambda$,
\[
\Pt^\Neu_d(\lambda) \ge w_d\lambda^d.
\]
It will be convenient to set
\begin{equation}\label{eq:Rd}
\mathcal{R}_d =  \sqrt{\frac{2}{d}} \frac{\Gamma\left(\frac{d}{2}+1\right)}{\Gamma\left(\frac{d}{2}+\frac{1}{2}\right)},
\end{equation}
and
\begin{equation}\label{eq:Q}
\mathcal{Q}_d(\lambda):=\frac{1}{\mathcal{R}_d w_d  \lambda^d}\left({\Pt^\Neu_d(\lambda) - w_d \lambda^d}\right),
\end{equation}
and to prove that $\mathcal{Q}_d(\lambda)>0$.

Specifically, we prove
\begin{theorem}\label{thm:A} 
Let $d\ge 3$.  Using notation \eqref{eq:lambdatau}, 
assume that 
\begin{equation}\label{eq:tau1}
\tau\ge 1.
\end{equation}
Then 
\begin{equation}\label{eq:Tineq}
\tau^3\mathcal{Q}_d\left(\tau^3 d^{3/2}\right) \ge \mathcal{T}_d(\tau),
\end{equation}
where the function $\mathcal{T}_d(\tau)$ is monotone increasing in both $\tau$ and $d$. Therefore, we have two possibilities.
\begin{enumerate}[\rm(i)]
\item If $\mathcal{T}_d(1)\le 0$, then the function $\mathcal{T}_d(\tau)$ has a unique positive root $\tau^*_d\ge 1$, and $\mathcal{T}_d(\tau)>0$ for $\tau>\tau^*_d$. Moreover, $\tau^*_d$ is monotone decreasing in $d$. This situation occurs if $3\le d\le 60$, see Figure \ref{fig:totalpictsmall}.
\item If $\mathcal{T}_d(1)> 0$, then $\mathcal{T}_d(\tau)>0$ for $\tau\ge 1$. This happens for all $d\ge 61$, in which case we set $\tau_d^*:=1$.
\end{enumerate}
Additionally, if $d\ge 13$, then $\tau^*_d<\frac{28}{25}=:\tau^*$.

Therefore, P\'olya's inequality \eqref{eq:Polyatau} holds for all $\tau>\tau^*_d$.
\end{theorem}

Some typical plots of functions $\mathcal{T}_d(\tau)$ are shown in Figure \ref{fig:Tdtypical}.

\begin{figure}[htb]
\centering
\includegraphics{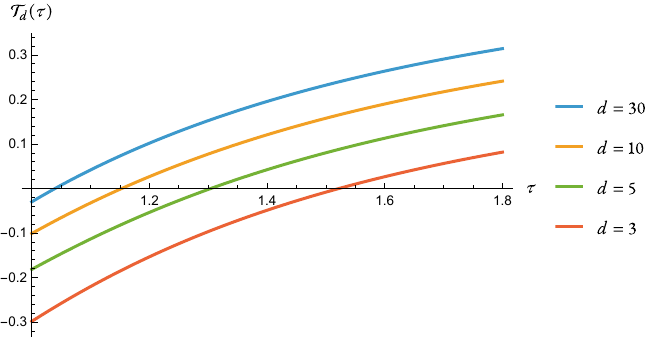}
\caption{$\mathcal{T}_d(\tau)$  as a function of $\tau$ for some values of $d$.\label{fig:Tdtypical}}
\end{figure}

We will prove Theorem \ref{thm:A} via a sequence of statements. 
We define
\begin{equation}\label{eq:defofm0}
\mathcal{M}=\mathcal{M}_{d}(\lambda):=
\min\left\{m\in\mathbb{N}_0: G_{\lambda}\left(m+\frac{d}{2}-1\right)<\frac{1}{4}\right\},
\end{equation}
and 
\begin{equation}\label{eq:defofell0}
\mathcal{L}=\mathcal{L}_{d}(\lambda):=\mathcal{M}_{d}(\lambda)+\frac{d}{2}-1,\qquad\text{with }G_{\lambda}\left(\mathcal{L}\right)<\frac{1}{4}.
\end{equation}
We also define the functions
\begin{equation}\label{eq:Sd}
S_d(z):=z^{d-3}\left(z^2-\frac{d-1}{2}z- \frac{(d-1)(d-2)(d-6)}{24}\right)
\end{equation}
and
\begin{equation}\label{eq:Sdp}
S_{d,+}(z):=z^{d-3}\left(z^2-\frac{d-1}{2}z- \frac{(d-1)(d-2)(d-6)_+}{24}\right),
\end{equation}
where $n_+:=\max\{n,0\}$, $n\in\mathbb{Z}$.

The first step is
\begin{prop}\label{prop:Qbound123}
\[
\mathcal{Q}_d(\lambda)\ge -Q_{d,1}(\lambda) - Q_{d,2}(\lambda) + Q_{d,3}(\lambda),
\]
where
\begin{align}
Q_{d,1}(\lambda)&:=\frac{d\mathcal{R}_d}{2} \lambda^{-1} + \frac{d^2(d-6)_+}{24\mathcal{R}_d}  \lambda^{-2},\label{eq:Q1}\\
Q_{d,2}(\lambda)&:=\frac{\sqrt{2 \pi } d^{5/2}}{\lambda^d}\int_{\mathcal{L}}^{\lambda} z^{d-2} G_{\lambda}(z)\, \dr z,\label{eq:Q2}\\
Q_{d,3}(\lambda)&:=\frac{1}{4\mathcal{R}_d w_d  \lambda^d}\,\sum_{n=0}^{\mathcal{M}-1}\kappa_{d,n}.\label{eq:Q3}
\end{align}
\end{prop}

We will now provide a sequence of estimates for $Q_{d,j}\left(\tau^3 d^{3/2}\right)$, $j=1,2,3$, assuming condition \eqref{eq:tau1}.

\begin{prop}\label{prop:Q1b} 
We have 
\begin{equation}\label{eq:T1}
\tau^3 Q_{d,1}\left(\tau^3 d^{3/2}\right)\le t_{1,0}(d)+t_{1,3}(d)\tau^{-3}=:\mathcal{T}_{d,1}(\tau),
\end{equation}
where
\[
t_{1,0}(d):=\frac{32 d^2+8 d+1}{64 d^{5/2}},\qquad t_{1,3}(d):=\frac{(d-6)_+}{24 d+6}.
\]
Moreover, the function $\mathcal{T}_{d,1}(\tau)$ is monotone decreasing in $\tau$ and, subject to condition \eqref{eq:tau1}, is also monotone decreasing in $d$.
\end{prop}

\begin{prop}\label{prop:Q2b} 
Assuming that  condition \eqref{eq:tau1} holds, we have $\mathcal{L}<\lambda$ and
\begin{equation}\label{eq:T2}
\begin{split}
\tau^3 Q_{d,2}\left(\tau^3 d^{3/2}\right)&\le t_{2,2}\tau^{-2}+t_{2,4}(d)\tau^{-4} +t_{2,6}(d)\tau^{-6}+t_{2,8}(d)\tau^{-8}\\
&= t_{2,2}\left(\tau^{-2}+\alpha_{4}(d)\tau^{-4} +\alpha_{6}(d)\tau^{-6}+\alpha_{8}(d)\tau^{-8}\right)=:\mathcal{T}_{d,2}(\tau),
\end{split}
\end{equation}
where 
\[
t_{2,2}=\frac{3^{2/3}\pi^{7/6}}{20\cdot 2^{5/6}},
\]
and
\[
\begin{split}
\alpha_{4}(d):=\frac{t_{2,4}(d)}{t_{2,2}}&=-\frac{2^{1/6}3^{2/3}\pi^{2/3}\left(10\sqrt2 d-3\pi-14\sqrt2\right)}{112d},\\
\alpha_{6}(d):=\frac{t_{2,6}(d)}{t_{2,2}}&=\frac{6^{1/3}\pi^{4/3}(d-2)\left(10d-3\sqrt2\pi-18\right)}{384d^2},\\
\alpha_{8}(d):=\frac{t_{2,8}(d)}{t_{2,2}}&=\frac{27\pi^2(d-3)(d-2)(\sqrt2\pi-4)}{11264d^3}.
\end{split}
\]
Moreover, under the same condition, $\mathcal{T}_{d,2}(\tau)$  is monotone decreasing both in $\tau$ and $d$.
\end{prop}

\begin{prop}\label{prop:Q3b} 
Assume that  condition \eqref{eq:tau1} holds, and define
\begin{equation}\label{eq:omega}
\omega:=2^{-7/3}(3\pi)^{2/3}\approx 0.8853.
\end{equation}
We have 
\begin{equation}\label{eq:T3}
\tau^3 Q_{d,3}\left(\tau^3 d^{3/2}\right)\ge  \frac{1}{2} \sqrt{\frac{\pi }{2}}\er^{-\omega\tau^{-2}} \left(1-\frac{d-1}{2 \sqrt{d} \left(d \tau ^3-\tau  \omega \right)}-\frac{(d-1)(d-2)(d-6)}{24 d \left(d \tau ^3-\tau  \omega \right)^2}\right)
=:\mathcal{T}_{d,3}(\tau),
\end{equation}
where $\mathcal{T}_{d,3}(\tau)$  is monotone increasing both in $\tau\ge 1$ and $d\ge 3$.
\end{prop}

Propositions \ref{prop:Qbound123}--\ref{prop:Q3b} are proved below. Assuming their validity, we proceed now to the proof of Theorem \ref{thm:A} proper. 

First of all, we set
\[
\mathcal{T}_d(\tau):= -\mathcal{T}_{d,1}(\tau) - \mathcal{T}_{d,2}(\tau)+ \mathcal{T}_{d,3}(\tau).
\]
By Propositions \ref{prop:Qbound123}--\ref{prop:Q3b}, we get \eqref{eq:Tineq} with the desired monotonicity properties. As 
\[
\mathcal{T}_d(\tau) = \frac{1}{64} \left(16 \sqrt{2 \pi }-\frac{8 d (4 d+1)+1}{d^{5/2}}\right)+O\left(\tau^{-2}\right)\qquad\text{as }\tau\to\infty,
\]
it is eventually positive. The explicit verified rational approximations show that 
\[
\mathcal{T}_{60}(1)<-\frac{29648872838231958923}{75557863725914323419136}<0,
\]
\[
\mathcal{T}_{61}(1)>\frac{26358269750698667}{4722366482869645213696}>0,
\]
and therefore $\tau_d^*>1$ for $d\le 60$ and $\tau_d^*=1$ for $d\ge 61$. Similarly, 
\[
\mathcal{T}_{13}\left(\frac{28}{25}\right)>\frac{7605667427504885135}{9444732965739290427392}>0,
\]
and therefore $\tau_d^*<\frac{28}{25}$ for all $d\ge 13$.

Finally, a combination of \eqref{eq:modneumcount}, \eqref{eq:latticecomparison}, \eqref{eq:Q}, and \eqref{eq:Tineq} implies \eqref{eq:Polyatau}.

\subsection{Proof of Proposition \ref{prop:Qbound123}}

We recall

\begin{lemma}\label{lem:sumNeu}
Let $a\in\mathbb{Z}$, $b\in\mathbb{R}$, with $b>a$, and let $g:[a,b]\to\mathbb{R}$ be a non-negative strictly decreasing convex function such that $g(a)\ge \frac{1}{4}$, $g(b)=0$, and $g$ is $\mathrm{Lip}_{\frac{1}{2}}$, that is $|g(z)-g(w)|\le \frac{1}{2}|z-w|$ for all $z,w\in[a, b]$. Let
\[
\mathbf{M}=\mathbf{M}_g
:=1+\max\left\{m\in\{a,\dots,\entire{b}\}: g(m)\ge \frac{1}{4}\right\}=1+\entire{g^{-1}\left(\frac{1}{4}\right)},
\]
and assume that $\mathbf{M}\le b$. 
Then 
\begin{equation}\label{eq:gsumbound}
\sum_{m=a}^{\entire{b}} \entire{g(m)+\frac{3}{4}}\ge \int_a^{\mathbf{M}} g(z)\,\dr z+\frac{\mathbf{M}-a}{4}.
\end{equation}
\end{lemma}

\begin{remark}\label{rem:M0} 
Lemma \ref{lem:sumNeu} is, up to a change of an independent variable, the same as \cite[Lemma 3.4]{FLPS-AB} and \cite[Theorem 6.1]{FLPS}. We note that in the left-hand side of \eqref{eq:gsumbound} one can replace the upper sum limit $m=\entire{b}$ by $m=\mathbf{M}-1$. Additionally, if $g(a)< \frac{1}{4}$, then \eqref{eq:gsumbound} remains (trivially) valid if we set $\mathbf{M}:=a$. 
\end{remark}

According to our definition \eqref{eq:defofm0}, 
\[
\mathcal{M}=\mathbf{M}_{z\mapsto G_{\lambda}\left(z+\frac{d}{2}-1\right)}.
\]
We have, taking in \eqref{eq:latticecomparison} $\kappa_{d,-1}=0$,
\[
\Pt^\Neu_d(\lambda)=\sum_{n=0}^{\entire{\lambda-\frac{d}{2}+1}}\left(\kappa_{d,n}-\kappa_{d,n-1}\right) \sum_{m=n}^{\entire{\lambda-\frac{d}{2}+1}}\entire{G_{\lambda}\left(m+\frac{d}{2} - 1\right)+\frac 34}.
\]
In order to estimate each inner sum, we apply Lemma \ref{lem:sumNeu} with $g(z)=G_{\lambda}\left(z+\frac{d}{2}-1\right)$, $a=n$, and $b=\lambda-\frac{d}{2}+1$, so that $g(b)=G_\lambda(\lambda)=0$,
which gives, after the change of variables $t=z+\frac{d}{2}-1$, with account of \eqref{eq:defofell0}, and changing the upper sum limits from $\entire{\lambda-\frac{d}{2}+1}$ to  $\mathcal{M}-1$ in accordance with Remark \ref{rem:M0},
\begin{equation}\label{eq:dredN1}
\begin{split}
\Pt^\Neu_d(\lambda)&\ge \sum_{n=0}^{\mathcal{M}-1}\left(\kappa_{d,n}-\kappa_{d,n-1}\right)\left(\int_{n}^{\mathcal{M}}G_{\lambda}\left(z+\frac{d}{2}-1\right)\, \dr z+\frac{\mathcal{M}-n}{4}\right)\\
&=\sum_{n=0}^{\mathcal{M}-1}\left(\kappa_{d,n}-\kappa_{d,n-1}\right)\left(\int_{n+\frac{d}{2}-1}^{\mathcal{L}}G_{\lambda}(t)\, \dr t+\frac{\mathcal{M}-n}{4}\right).
\end{split}
\end{equation}

Define the constants
\[
z_m:=m+d/2-1, \qquad m\in\mathbb{N}_0,
\] 
and a piecewise-constant function 
\[
f_d^\Neu(z)=
\begin{cases} 0 &\qquad\text{if }z<z_0,\\
\kappa_{d,m} &\qquad\text{if }z\in\left[z_m,z_{m+1}\right),\quad m=\entire{z-\frac{d}{2}+1}\in\mathbb N_0.
\end{cases}
\]

After trivial changes, we re-write \eqref{eq:dredN1} as
\begin{equation}\label{eq:toestimate}
\Pt^\Neu_d(\lambda)\ge\int_0^{\mathcal{L}}f^\Neu_d(z)G_{\lambda}(z)\, \dr z + \frac{1}{4}\sum_{n=0}^{\mathcal{M}-1}\kappa_{d,n},
\end{equation}

We also define
\[
F_d^\Neu(z):=\int_0^z f_d^\Neu(t)\,\dr t=
\begin{cases} 0 &\quad\text{if }z<z_0,\\
\left(\sum_{j=0}^{m-1}\kappa_{d,j}\right)+(z-z_m)\kappa_{d,m} &\quad\text{if }z\in\left[z_m,z_{m+1}\right), m=\entire{z-\frac{d}{2}+1}\in\mathbb N_0.
\end{cases}
\]

We will use an auxiliary lemma similar to \cite[Lemma 7.2]{FLPS}, see \S\ref{proof:lem:prodint} for the proof. 
\begin{lemma}\label{lem:prodint} 
Let $f$ be a locally integrable non-negative function on $[0,+\infty)$, and let $\utilde{F}\in C^1\left([0,+\infty)\right)$ with $\utilde{F}(0)=0$ be such that
\[
\utilde{F}(z)\le F(z):=\int_0^z f(t)\,\dr t\qquad\text{for all }z\ge 0.
\] 
Let $y>0$, and let $g\in C^1[0,y]$ be a decreasing function such that $g(y)\ge 0$. Then
\begin{equation}\label{eq:fgprod}
\int_0^y f(z)g(z)\,\dr z \ge \int_{0}^y \utilde{F}'(z) g(z)\,\dr z.
\end{equation}
\end{lemma}

Further on, we have

\begin{lemma}\label{lem:FdNbound} 
Let $d\ge 3$ and $z\ge 0$. Then
\[
F_d^\Neu(z)\ge \frac{2}{(d-1)!} S_{d,+}(z),
\]
see \eqref{eq:Sdp} for the definition of the right-hand side.
\end{lemma}

\begin{proof}
We first check that the inequality holds for $z\le \frac{d}{2}-1$, when $F_d^\Neu(z)\equiv 0$. 
As  the quadratic term in the definition of  $S_{d,+}(z)$ is non-positive both at $z=0$, where it equals $-\frac{(d-1)(d-2)(d-6)_+}{24}$, and at $z= \frac{d}{2}-1$, where it equals $\frac{1}{2}-\frac{d}{4}-\frac{(d-1)(d-2)(d-6)_+}{24}$, we conclude that 
\[
0=F_d^\Neu(z)\ge \frac{2}{(d-1)!} S_{d,+}(z),\qquad z\in\left[0, \frac{d}{2}-1\right].
\]

We will now show that 
\begin{equation}\label{eq:FdNbound2}
F_d^\Neu(z)\ge\frac{2\left(z-\frac{1}{2}\right)\Pi_{d-2}\left(z-\frac{d}{2}\right)}{(d-1)!}.
\end{equation}
for $z\ge \frac{d}{2}-1$, where $\Pi_n$ is defined by \eqref{eq:Pin}.  Firstly, if $z=z_m:=m+\frac{d}{2}-1$, $m\in\mathbb{N}_0$, then
\[
F_d^\Neu(z_m)=\sum_{k=0}^{m-1} \kappa_{d,k}=\frac{2\left(m-1+\frac{d-1}{2}\right)\Pi_{d-2}(m-1)}{(d-1)!}
\]
by \eqref{eq:kappadmsum} with $l=m-1$ and therefore inequality \eqref{eq:FdNbound2} becomes equality; note that both sides vanish when $m=0$.
For $z\in(z_m, z_{m+1})$, the function $F_d^\Neu(z)$ is linear in $z$, and the right-hand side of \eqref{eq:FdNbound2} can be re-written as 
\[
\frac{2}{(d-1)!}\times\left(\zeta+\frac{d-3}{2}\right)\times\zeta\times(\zeta+1)\cdots\times(\zeta+d-3), \qquad \zeta:=z-\frac{d}{2}+1,
\]
which is a polynomial in $\zeta$ with positive coefficients, and therefore convex. Thus, \eqref{eq:FdNbound2} holds for all $z\ge \frac{d}{2}-1$.
We now apply  \eqref{eq:Pibounds2} with $n=d-2\ge 1$ and $x=z-\frac{d}{2}\ge -1$,
and therefore with 
$x+\frac{n}{2}+1=z$, to the right-hand side of \eqref{eq:FdNbound2}, which yields (still for $z\ge \frac{d}{2}-1$),
\[
F_d^\Neu(z)\ge \frac{2}{(d-1)!} S_d(z)\ge \frac{2}{(d-1)!} S_{d,+}(z).
\]
\end{proof}

We now apply Lemma \ref{lem:prodint} with $f=f_d^\Neu$, $g=G_\lambda$, $F=F_d^\Neu$, $\utilde{F}=\frac{2}{(d-1)!}S_{d,+}$, and $y=\mathcal{L}$ to the integral in the right-hand side of \eqref{eq:toestimate}.
Note that 
\[
\frac{2}{(d-1)!} S'_{d,+}(z) =
\begin{cases}
\frac{2}{(d-2)!}z^{d-2}-\frac{1}{(d-3)!}z^{d-3}-\frac{(d-6)_+}{12(d-4)!}z^{d-4}\qquad&\text{if }d\ge 4,\\
2z-1,\qquad&\text{if }d=3,
\end{cases}
\]
is continuous on $[0,+\infty)$ for all $d\ge 3$. We get
\begin{equation}\label{eq:intbound1}
\begin{split}
\int_0^{\mathcal{L}} f^\Neu_d(z)G_{\lambda}(z)\, \dr z &\ge \frac{2}{(d-1)!} \int_0^{\mathcal{L}} S'_{d,+}(z) G_{\lambda}(z)\, \dr z\\
&=\frac{2}{(d-1)!} \int_0^{\lambda}  S'_{d,+}(z) G_{\lambda}(z)\, \dr z - \frac{2}{(d-1)!} \int_{\mathcal{L}}^{\lambda} S'_{d,+}(z) G_{\lambda}(z)\, \dr z.
\end{split}
\end{equation}
The first integral in the right-hand side of \eqref{eq:intbound1} can be evaluated explicitly by Lemma \ref{lem:GintW} with $l=d-2,d-3,d-4$:
\begin{equation}\label{eq:explint}
\begin{split}
\frac{2}{(d-1)!} \int_0^{\lambda} S'_{d,+}(z) G_{\lambda}(z)\, \dr z&=\frac{2}{(d-2)!}\int_0^{\lambda}z^{d-2}G_{\lambda}(z)\, \dr z - \frac{1}{2}\cdot \frac{2}{(d-3)!}\int_0^{\lambda}z^{d-3}G_{\lambda}(z)\,\dr z \\&\qquad - \frac{(d-6)_{+}}{24}\cdot\frac{2}{(d-4)!}\int_0^{\lambda}z^{d-4}G_{\lambda}(z)\, \dr z \\
&=w_d\lambda^d-\frac 12w_{d-1}\lambda^{d-1}-\frac{(d-6)_+}{24}w_{d-2}\lambda^{d-2}.
\end{split}
\end{equation}
For the second integral in the right-hand side of \eqref{eq:intbound1} we use the obvious bounds
\[
\frac{2}{(d-1)!} \int_{\mathcal{L}}^{\lambda} S'_{d,+}(z)  G_{\lambda}(z)\, \dr z  \le \frac{2}{(d-2)!}\int_{\mathcal{L}}^{\lambda} z^{d-2}  G_{\lambda}(z)\, \dr z < \frac{2d}{(d-1)!}\int_{\mathcal{L}}^{\lambda} z^{d-2}  G_{\lambda}(z)\, \dr z.
\]

Combining  \eqref{eq:toestimate}, \eqref{eq:intbound1}, and \eqref{eq:explint} with \eqref{eq:Q}, and using \eqref{eq:wratios} and \eqref{eq:factwratios}, we arrive at
\[
\mathcal{Q}_d(\lambda)\ge 
-\left(\frac{d\mathcal{R}_d}{2} \lambda^{-1} + \frac{d^2(d-6)_+}{24\mathcal{R}_d}  \lambda^{-2}\right) 
-\frac{\sqrt{2 \pi } d^{5/2}}{\lambda^d}\int_{\mathcal{L}}^{\lambda} z^{d-2} G_{\lambda}(z)\, \dr z
+\frac{1}{4\mathcal{R}_d w_d  \lambda^d}\sum_{n=0}^{\mathcal{M}-1}\kappa_{d,n},
\]
completing the proof of Proposition \ref{prop:Qbound123}.

\subsection{Proof of Proposition \ref{prop:Q1b}}
In order to prove \eqref{eq:T1}, we use the explicit expression \eqref{eq:Q1} and estimates \eqref{eq:Rdbounds} from Lemma \ref{lem:Gammaratio1}, immediately giving the result.

Monotonicity of $\mathcal{T}_{d,1}(\tau)$ in $\tau$ is obvious, for the proof of monotonicity in $d$, see \S\ref{proof:prop:Q1b}.

\subsection{Proof of Proposition \ref{prop:Q2b}}

We start with

\begin{lemma}\label{lem:estell} 
With $\mathcal{L}=\mathcal{L}_d(\lambda)$ defined by \eqref{eq:defofell0}, we have, for all $\lambda\ge d^{3/2}$,
\[
\lambda>\mathcal{L}\ge \lambda-\omega\lambda^{1/3}=:\utilde{\mathcal{L}}=\utilde{\mathcal{L}}(\lambda),
\]
where $\omega$ given by \eqref{eq:omega} is defined as  $\omega:=\left(4c_1\right)^{-2/3}$ with $c_1$  given by \eqref{eq:c1c2}.
\end{lemma}

\begin{proof}
The bound $\mathcal{L}<\lambda$ follows from definitions \eqref{eq:defofell0}, \eqref{eq:defofm0}, and the fact that $G_\lambda^{-1}\left(\frac{1}{4}\right)<\lambda-1$ for all $\lambda\ge 3^{3/2}>2$ by \cite[Lemma 4.8]{FLPS}. 

We also have, using  \eqref{eq:defofell0} and the lower bound in Lemma \ref{lem:Gbounds},
\[
c_1\lambda\left(1-\frac{\mathcal{L}}{\lambda}\right)^{3/2}<G_\lambda(\mathcal{L})<\frac{1}{4},
\]
from where the lower bound for $\mathcal{L}$ follows by solving the inequality.
\end{proof}

Set 
\begin{equation}\label{eq:Theta0}
\Theta=\Theta(\lambda):=1-\frac{\utilde{\mathcal{L}}}{\lambda}=\omega \lambda^{-2/3};
\end{equation} 
clearly, $\Theta(\lambda)\le 1$ as $\lambda\ge d^{3/2}>1$.

We use the upper bound on $G_\lambda(z)$ from Lemma \ref{lem:Gbounds} and the lower bound on $\mathcal{L}$ from Lemma \ref{lem:estell} to get
\[
\int_{\mathcal{L}}^{\lambda} z^{d-2} G_\lambda(z)\,\dr z \le \lambda \int_{\utilde{\mathcal{L}}}^{\lambda} z^{d-2} \left(c_1\left(1-\frac{z}{\lambda}\right)^{3/2}+c_2\left(1-\frac{z}{\lambda}\right)^{5/2}\right)\,\dr z.
\]

We make the change of variables $z=\lambda(1-\theta)$.
As the interval $z\in[\utilde{\mathcal{L}}(\lambda),\lambda]$ corresponds to $\theta\in[0,\Theta(\lambda)]\subset [0,1]$, we have
\begin{equation}\label{eq:in2b2}
\lambda \int_{\utilde{\mathcal{L}}}^{\lambda} z^{d-2} \left(c_1\left(1-\frac{z}{\lambda}\right)^{3/2}+c_2\left(1-\frac{z}{\lambda}\right)^{5/2}\right)\,\dr z 
= \lambda^d \int_0^{\Theta} (1-\theta)^{d-2}\left(c_1\theta^{3/2}+c_2\theta^{5/2}\right)\,\dr\theta.
\end{equation}
For integer $d\ge 3$ and $0\le \theta\le 1$, we have
\[ 
(1-\theta)^{d-2} \le 1-(d-2)\theta +\frac{(d-2)(d-3)}{2}\theta^2.
\]
Hence, substituting this bound to the right-hand side of \eqref{eq:in2b2},  integrating, and multiplying by the factor $\frac{\sqrt{2 \pi } d^{5/2}}{\lambda^d}$ appearing in \eqref{eq:Q2}, we obtain
\[
\begin{split}
Q_{d,2}(\lambda)\le \sqrt{2 \pi } d^{5/2} &\left(\frac{2}{5} c_1 \Theta ^{5/2}
-\frac{2}{7} \Theta ^{7/2} \left(c_1 (d-2)-c_2\right)\right.\\
&\left.+\frac{1}{9} (d-2) \Theta ^{9/2} \left(c_1 (d-3)-2 c_2\right)
+\frac{1}{11} c_2 (d-3) (d-2) \Theta ^{11/2}
\right).
\end{split}
\]
Substituting constants $c_1$ and $c_2$ from \eqref{eq:c1c2} and $\Theta$ from \eqref{eq:Theta0}, and switching from $\lambda$ to  $\tau$ in accordance with \eqref{eq:lambdatau}, we arrive, after some simplifications, at \eqref{eq:T2}.

For the proof of monotonicity of $\mathcal{T}_{d,2}(\tau)$ in both $\tau$ and $d$, subject to condition $\tau\ge 1$, see \S\ref{proof:prop:Q2b}.

\subsection{Proof of Proposition \ref{prop:Q3b}}

To start with, we apply \eqref{eq:kappadmsum} with $l=\mathcal{M}-1$ to get 
\begin{equation}\label{eq:sum1}
\sum_{j=0}^{\mathcal{M}-1} \kappa_{d,j}=\frac{2\left(\mathcal{M}-1+\frac{d-1}{2}\right)}{(d-1)!}\Pi_{d-2}(\mathcal{M}-1).
\end{equation}
We now estimate the right-hand side of \eqref{eq:sum1} using Lemma \ref{lem:Pi3} with $x=\mathcal{M}-1$, $n=d-2$, and therefore with $x+\frac{n}{2}+1=\mathcal{M}+\frac{d}{2}-1=\mathcal{L}$, giving, taking into account \eqref{eq:factwratios},
\[
\begin{split}
Q_{d,3}(\lambda)=\frac{1}{4\mathcal{R}_d w_d  \lambda^d}\sum_{j=0}^{\mathcal{M}-1} \kappa_{d,j}&\ge \frac{1}{2} \sqrt{\frac{\pi }{2}} d^{3/2} \lambda^{-d}\left(\mathcal{L}^{d-1} -\frac{d-1}{2}\mathcal{L}^{d-2} - \frac{(d-1)(d-2)(d-6)}{24}\mathcal{L}^{d-3}\right)\\
&= \frac{1}{2} \sqrt{\frac{\pi }{2}} d^{3/2} \lambda^{-d} S_d(\mathcal{L})
\end{split}
\]

We use the following result, see \S\ref{proof:lem:monder} for the proof.
\begin{lemma}\label{lem:monder}
Let $d\ge 3$ and let $\lambda\ge d^{3/2}$ (this condition is not necessary but sufficient). Then $S'_d(z)>0$ for all $z\ge \utilde{\mathcal{L}}(\lambda)$.
\end{lemma}

As a consequence, 
\[
\begin{split}
Q_{d,3}(\lambda) &\ge \frac{1}{2} \sqrt{\frac{\pi }{2}} d^{3/2} \lambda^{-d} S_d\left(\utilde{\mathcal{L}}\right) =  \frac{1}{2} \sqrt{\frac{\pi }{2}} d^{3/2} \lambda^{-d} S_d\left(\lambda(1-\Theta(\lambda))\right)\\
&= \frac{1}{2} \sqrt{\frac{\pi }{2}} d^{3/2} \lambda^{-1} \left(1-\Theta(\lambda)\right)^{d-1} \left(1-\frac{d-1}{2\lambda(1-\Theta(\lambda))} - \frac{(d-1)(d-2)(d-6)}{24\lambda^2(1-\Theta(\lambda))^2}\right)
\end{split}
\]
By \eqref{eq:Theta0}, $\Theta\left(\tau^3 d^{3/2}\right)=\frac{\omega}{d\tau^2}<\frac{1}{d}\le \frac{1}{3}$ for $\tau\ge 1$, and therefore by Lemma \ref{lem:log} with $x=\Theta(\lambda)$,
\[
\left(1-\Theta\right)^{d-1} = \er^{(d-1)\log \left(1-\Theta\right)} \ge\er^{(d-1)\left(-\Theta-\frac{\Theta^2}{2(1-\Theta)}\right)} = \er^{-d\Theta+\frac{\Theta(2-(d+1)\Theta)}{2(1-\Theta)}}>\er^{-d\Theta}=\er^{-\omega\tau^{-2}},
\]
where we have used $\Theta \le \frac 1d < \frac{2}{d+1}$ in the last inequality. 
Thus,
\[
\tau^3 Q_{d,3}\left(\tau^3 d^{3/2}\right)\ge \frac{1}{2} \sqrt{\frac{\pi }{2}}\er^{-\omega\tau^{-2}} \left(1-\frac{d-1}{2 \sqrt{d} \left(d \tau ^3-\tau  \omega \right)}-\frac{(d-1)(d-2)(d-6)}{24 d \left(d \tau ^3-\tau  \omega \right)^2}\right),
\]
proving the first part of Proposition  \ref{prop:Q3b}. For the proof of monotonicity in $\tau\ge 1$ and $d$, see \S\ref{proof:prop:Q3b}.

\section{Low $\lambda$ and variational estimates}\label{sec:lowlambda}

\subsection{Summary}

In this section, using variational techniques, we establish

\begin{theorem}\label{thm:low}
Let $d\ge 3$. P\'olya's inequality \eqref{eq:Polyatau} holds for all $\tau\in\left[0, \tau^*\right]$, where $\tau^*=\frac{28}{25}$.
\end{theorem}

Theorem \ref{thm:low} will follow from Propositions \ref{prop:verylow}, \ref{prop:varsummary}, \ref{prop:constant1}, \ref{prop:constant2}, and \ref{prop:othertest}.

\subsection{Very low $\lambda$}

Before proceeding to variational estimates proper, we remark that $\mu_1\left(\ball{d}\right)=0$ and  
$\mu_2\left(\ball{d}\right) \le d+2 \le w_d^{-\frac{2}{d}}$, where the first inequality follows from the variational principle by taking a trial function $x_1$, and the second one is an elementary calculation.
As for the ball in dimension $d$ we have 
\[
\mu_2\left(\ball{d}\right)=\mu_3\left(\ball{d}\right)=\dots=\mu_{d+1}\left(\ball{d}\right),
\]
we can guarantee that 
\[
\mathcal{N}^\Neu_{\ball{d}}(\lambda)\ge
\begin{cases}
1, \qquad& 0\le \lambda< w_d^{-\frac{1}{d}},\\
d+1, \qquad& \lambda\ge w_d^{-\frac{1}{d}}.
\end{cases}
\]
and therefore $\mathcal{N}^\Neu_{\ball{d}}(\lambda) \ge w_d \lambda^d$ for all $\lambda$ such that $w_d \lambda^d \le d+1$, that is for 
\[
0\le \lambda\le 2 (d+1)^{\frac{1}{d}}\Gamma \left(\frac{d}{2}+1\right)^{\frac{2}{d}}=:\lambda_d^\#.
\]
After changing the parameter  according to \eqref{eq:lambdatau}, we arrive at
\begin{prop}\label{prop:verylow} 
P\'olya's inequality
\[
\mathcal{N}^\Neu_{\ball{d}}\left(\tau^3 d^{3/2}\right)-w_d \tau^{3d} d^{3d/2}\ge 0
\]
holds for all 
\[
\tau\le \frac{2^{\frac{1}{3}}(d+1)^{\frac{1}{3d}}\,\Gamma\left(\frac{d}{2}+1\right)^{\frac{2}{3d}}}{\sqrt{d}}  =: \tau_d^{\#}.
\]
\end{prop}

\begin{remark} In what follows, we will only have to use Proposition \ref{prop:verylow} for $d=3$, in which case we note that $ \tau_3^{\#}=\frac{(2 \pi )^{1/9}}{3^{5/18}}>\frac{17}{20}$.
\end{remark}

\subsection{Variational approach: the algorithm}
Let us count, in the right-hand side of \eqref{eq:defofneumcount}, only the \emph{first} zeros $p'_{d,m,1}$, which would obviously give another lower bound for  $\mathcal N^\Neu_{\mathbb B^d}(\lambda)$. With account of $p'_{d,0,1}=0$, we get 
\begin{equation}\label{eq:firstcount}
\mathcal{N}^\Neu_{\mathbb B^d}(\lambda) \ge 1+\sum_{m\in\mathbb{N}:  p'_{d,m,1}\le\lambda}\kappa_{d,m}=:\mathcal{N}^\Neu_{d, 1\mathrm{st}}(\lambda).
\end{equation} 
We rewrite the variational principle \eqref{eq:varpr}, setting $\eta:=m+\frac{d}{2}-1$, as
\[
\mathcal{A}_d\left(\eta\right):=A_d\left(\eta-d/2+1\right)=\frac{\mathcal{I}_{d,1}+\left(\eta^2-\left(\frac{d}{2}-1\right)^2\right)\mathcal{I}_{d,2}}{\mathcal{I}_{d,3}},
\] 
where
\[
\mathcal{I}_{d,1} := \int_0^1 |\rho'(x)|^2 x^{d-1}\,\dr x,\qquad \mathcal{I}_{d,2} :=  \int_0^1|\rho(x)|^2x^{d-3}\,\dr x, \qquad \mathcal{I}_{d,3} :=  \int_0^1|\rho(x)|^2x^{d-1}\, \dr x.
\]
Assume for now that $\rho$ does not depend upon $m$. Then $A_d(m)$ is monotone increasing in $m$. Denote, for brevity,
\begin{equation}\label{eq:ab} 
\mathfrak{a}_d:=\frac{\mathcal{I}_{d,3}}{\mathcal{I}_{d,2}}, \qquad \mathfrak{b}_d:=\frac{\mathcal{I}_{d,1}}{\mathcal{I}_{d,2}}.
\end{equation}

Let, for $\lambda>0$, $\mathfrak{X}=\mathfrak{X}_d(\lambda)\in\mathbb{R}$ be the positive solution  (or zero if there are no positive solutions) of the equation
\[
A_d\left(\mathfrak{X}\right) = \lambda^2.
\]
This can be solved explicitly: if $\mathfrak{Y}=\mathfrak{Y}(\lambda)\ge\frac{d}{2}-1$ solves $\mathcal{A}_d\left(\mathfrak{Y}\right)=\lambda^2$, then
\begin{equation}\label{eq:Y2defn}
\mathfrak{Y}^2 := \frac{\lambda^2  \mathcal{I}_{d,3} - \mathcal{I}_{d,1}}{\mathcal{I}_{d,2}}+\left(\frac{d}{2}-1\right)^2=\lambda^2 \mathfrak{a}_d - \mathfrak{b}_d +\left(\frac{d}{2}-1\right)^2,
\end{equation}
and
\[
\mathfrak{X} = 
\begin{cases}
\mathfrak{Y}-\left(\frac{d}{2}-1\right)=\sqrt{\frac{\lambda^2  \mathcal{I}_{d,3} - \mathcal{I}_{d,1}}{\mathcal{I}_{d,2}}+\left(\frac{d}{2}-1\right)^2}-\left(\frac{d}{2}-1\right)\qquad&\text{if }\lambda^2\ge\frac{\mathcal{I}_{d,1}}{\mathcal{I}_{d,3}}=\frac{\mathfrak{b}_d}{\mathfrak{a}_d},\\
0\qquad&\text{otherwise}.
\end{cases}
\]
Clearly, $\mathfrak{X}$ is an increasing function of $\lambda$, and is therefore invertible as long as $\mathfrak{X}>0$, in which case we have
\begin{equation}\label{eq:lambdaab}
\lambda^2=\frac{1}{\mathfrak{a}_d}\left(\mathfrak{X}\left(\mathfrak{X}+d-2\right)+\mathfrak{b}_d\right).
\end{equation}
 
Let
\[
\mathfrak{M}:=\entire{\mathfrak{X}}\ge \mathfrak{X}-1.
\]
Then, using \eqref{eq:kappadmsum}, \eqref{eq:factwratios}, and \eqref{eq:Rdbounds},
\begin{equation}\label{eq:N1st}
\begin{split}
\frac{1}{w_d}\mathcal{N}^\Neu_{d, 1\mathrm{st}}(\lambda) &\ge \frac{1}{w_d}\sum_{m=0}^{\mathfrak{M}} \kappa_{d,m} =\sqrt{2\pi}\mathcal{R}_d d^{3/2} \left(\mathfrak{M}+\frac{d-1}{2}\right)\Pi_{d-2}(\mathfrak{M})\\
&\ge \sqrt{2\pi}\left(1+\frac{1}{4d}\right)d^{3/2} \left(\mathfrak{X}+\frac{d-3}{2}\right)\times \mathfrak{X}\times (\mathfrak{X}+1)\times \dots\times  (\mathfrak{X}+d-3)=:\utilde{\mathcal{N}}_{d, 1\mathrm{st}}(\lambda).
\end{split}
\end{equation}
For brevity, set
\[
\mathfrak c_d
:=
\sqrt{2\pi}
\left(1+\frac1{4d}\right),
\]
\begin{equation}\label{eq:Fd}
\mathfrak{F}_d(\lambda):=\frac{\utilde{\mathcal{N}}_{d, 1\mathrm{st}}(\lambda)}{\lambda^d}=
\frac{\mathfrak{c}_d d^{3/2} \left(\mathfrak{X}(\lambda)+\frac{d-3}{2}\right)\times \mathfrak{X}(\lambda)\times \dots\times  (\mathfrak{X}(\lambda)+d-3)}{\lambda^d}.
\end{equation}

We therefore arrive at
\begin{prop}\label{prop:varsummary} 
Suppose that for a given $\tau>0$ and a given test function $\rho\in H^1\left((0,1), x^{d-1}\,\dr x\right)$ independent of $m$, we have
\begin{equation}\label{eq:N1staim}
\mathfrak{F}_d\left(\tau^3 d^{3/2}\right)-1 > 0.
\end{equation}
Then, with $\lambda=\tau^3 d^{3/2}$,
\[
\mathcal{N}^\Neu_{\mathbb B^d}(\lambda) > w_d \lambda^d.
\] 
\end{prop}

\begin{proof}
We have, by \eqref{eq:firstcount}, \eqref{eq:N1st}, and \eqref{eq:N1staim},
\[
\mathcal{N}^\Neu_{\mathbb B^d}(\lambda) \ge \mathcal{N}^\Neu_{d, 1\mathrm{st}}(\lambda)\ge w_d\utilde{\mathcal{N}}_{d, 1\mathrm{st}}(\lambda) > w_d \lambda^d.
\]
\end{proof}

We aim, for a given $d\ge 3$ and a test function $\rho$, to find a range of values of the  parameter $\tau$ such that \eqref{eq:N1staim} holds.
Before proceeding to consider specific test functions $\rho$, we make additional observations. 
Using \eqref{eq:lambdaab}, we can re-write \eqref{eq:Fd} as 
$\mathfrak{F}_d(\lambda) = \mathfrak{Z}_d(\mathfrak{X}(\lambda))$, where
\[
 \mathfrak{Z}_d(\mathfrak{X}) :=  \mathfrak{c}_d d^{3/2} \mathfrak{a}_d^{\frac{d}{2}}\frac{\left(\mathfrak{X}+\frac{d-3}{2}\right) \Pi_{d-2}\left(\mathfrak{X}-1\right)}{\left(\mathfrak{X}\left(\mathfrak{X}+d-2\right)+\mathfrak{b}_d\right)^{d/2}}.
\]
Applying Lemma \ref{lem:Pi3} to the right-hand side of \eqref{eq:N1st},
we deduce that with $\lambda=\tau^3 d^{3/2}$ we have
\begin{equation}\label{eq:N1stbound}
\mathfrak{F}_d\left(\tau^3 d^{3/2}\right) \ge \tau^{-3} \mathfrak{c}_d \left(\frac{\mathfrak{Y}(\lambda)}{\lambda}\right)^{d-1} \left(1-\frac{d-1}{2\mathfrak{Y}(\lambda)}-\frac{(d-1)(d-2)(d-6)}{24\left(\mathfrak{Y}(\lambda)\right)^2}\right).
\end{equation}

\subsection{A constant test function}

We start with $\rho(x)\equiv 1$. Then 
\begin{equation}
\mathcal{I}_{d,1} = 0, \qquad\mathcal{I}_{d,2}=\frac{1}{d-2}, \qquad\mathcal{I}_{d,3}=\frac{1}{d},\qquad \mathfrak{a}_d=\frac{d-2}{d},\qquad \mathfrak{b}_d=0,
\end{equation}
and
\begin{equation}\label{eq:Xlambda1}
\mathfrak{X}(\lambda)=\sqrt{\frac{(d-2) \lambda ^2}{d}+\left(\frac{d}{2}-1\right)^2}-\left(\frac{d}{2}-1\right).
\end{equation}
In Figure \ref{fig:varone},  the left-hand side of \eqref{eq:N1staim} 
is plotted as a function of $\tau$.

\begin{figure}[htb]
\centering
\includegraphics{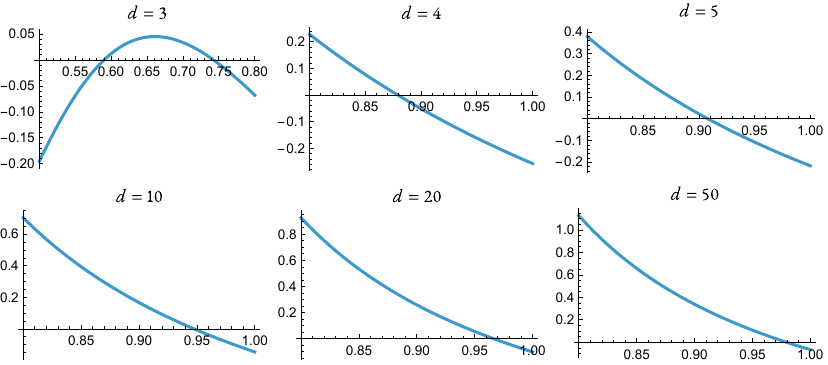}
\caption{Plots of the numerically-evaluated left-hand side of \eqref{eq:N1staim} as a function of $\tau$ for the constant test function. We are looking at the intervals of positivity. \label{fig:varone}}
\end{figure}

\begin{lemma}\label{lem:constantmonotone}
Let $\rho(x)\equiv 1$ and $d\ge4$. Then the function $\mathfrak{F}_d(\lambda)$ is strictly decreasing in the variable $\lambda\in(0,\infty)$. 
\end{lemma}

For the proof, see \S\ref{proof:constant}.

We now observe that when $\lambda\to 0^+$, the Taylor expansion of \eqref{eq:Xlambda1} gives $\mathfrak X(\lambda)=\frac{\lambda^2}{d}+O(\lambda^4)$.
Hence, for $d\ge4$,
\[
\mathfrak{F}_d(\lambda)
\sim
\mathfrak{c}_d d^{1/2}
\left(\frac{d-3}{2}\right)(d-3)!\,\lambda^{2-d}\to+\infty
\qquad\text{as }\lambda\to0^+.
\]
On the other hand, as $\lambda\to+\infty$, $\mathfrak X(\lambda) = \sqrt{\frac{d-2}{d}} \lambda+O(1)$.
Thus,
\[
\mathfrak{F}_d(\lambda)
\sim
\mathfrak{c}_d d^{3/2}\left(\frac{d-2}{d}\right)^{\frac{d-1}{2}} \lambda^{-1}
\to 0
\qquad\text{as }\lambda\to+\infty.
\]
Combining these limiting behaviours and Lemma \ref{lem:constantmonotone}, we obtain

\begin{prop}\label{prop:constant1}
Let $d\ge 4$. Then there exists a unique $\tau^\circ_d>0$ such that $\mathfrak{F}_d\left({\tau_d^\circ}^3 d^{3/2}\right)=1$, and  $\mathfrak{F}_d\left(\tau^3 d^{3/2}\right)>1$ for all $\tau\in\left(0,  \tau_d^\circ\right)$.
\end{prop}

A practical lower bound for $\tau^\circ_d$ is given by
\begin{prop}\label{prop:constant2}
For any $d\ge 4$, 
\[
\tau^\circ_d > \frac{17}{20}=:\tau^\circ.
\]
\end{prop}

\begin{proof} It is enough, by monotonicity, to show that $\mathfrak{F}_d\left({\tau^\circ}^3 d^{3/2}\right)>1$. We set
\[
\lambda^\circ : = {\tau^\circ}^3 d^{3/2}, \qquad \mathfrak{Y}^\circ =  \mathfrak{Y}\left(\lambda^\circ\right) = \sqrt{\frac{(d-2) {\lambda^\circ}^2}{d}+\left(\frac{d}{2}-1\right)^2},
\]
substitute these values into \eqref{eq:N1stbound}, and estimate the result term-by-term.

Firstly, we have 
\begin{equation}\label{eq:cflower}
\mathfrak{c}_d\ge \sqrt{2\pi}>\frac{193}{77}.
\end{equation}

Secondly,
\[
\left(\frac{\mathfrak{Y}^\circ}{\lambda^\circ}\right)^{d-1} = \left(1-\frac{2}{d}\right)^{\frac{d-1}{2}} \left(1+\frac{d-2}{4{\tau^\circ}^6 d^2}\right)^{\frac{d-1}{2}},
\]
and we estimate the two factors separately. For the first one, by Lemma \ref{lem:log} and \eqref{eq:loglower}, 
\[
\begin{split}
\left(1-\frac{2}{d}\right)^{\frac{d-1}{2}} &= \exp\left(\frac{d-1}{2} \log\left(1-\frac{2}{d}\right)\right) \ge \exp\left(-\frac{(d-1)^2}{d (d-2)}\right)\\\
&=\exp\left(-1-\frac{1}{d (d-2)}\right)\ge \exp\left(-1+\log\left(1-\frac{1}{d (d-2)}\right)\right)=\er^{-1} \left(1-\frac{1}{d(d-2)}\right),
\end{split}
\]
and for the second one, 
\[
\left(1+\frac{d-2}{4{\tau^\circ}^6 d^2}\right)^{\frac{d-1}{2}} \ge 1+\frac{(d-1) (d-2)}{8 d^2 {\tau^\circ}^6}.
\]

Finally, as $\mathfrak{Y}^\circ \ge {\tau^\circ}^3 d \sqrt{d-2}$,
\begin{equation}\label{eq:factor3lower}
-\frac{d-1}{2\mathfrak{Y}^\circ}\ge -\frac{d-1}{ 2d \sqrt{d-2}{\tau^\circ}^3},\qquad -\frac{(d-1)(d-2)(d-6)}{24{\mathfrak{Y}^\circ}^2}\ge -\frac{(d-1)(d-6)_+}{24 {\tau^\circ}^6 d^2} \ge -\frac{1}{24 {\tau^\circ}^6},
\end{equation}
the latter bound being trivial for $d\le 6$.

Substituting bounds \eqref{eq:cflower}--\eqref{eq:factor3lower}  into \eqref{eq:N1stbound}, we obtain
\[
\mathfrak{F}_d\left({\tau^\circ}^3 d^{3/2}\right) \ge 
\frac{193\er^{-1}}{77{\tau^\circ}^3}\left(1-\frac{1}{d(d-2)}\right) \left(1+\frac{(d-1) (d-2)}{8 d^2 {\tau^\circ}^6}\right) \left(1-\frac{d-1}{2 d \sqrt{d-2} {\tau^\circ}^3}-\frac{1}{24 {\tau^\circ}^6}\right),
\]
where each factor in the right-hand side is monotone increasing in $d$. Evaluating at $d=7$ gives
\[
\mathfrak{F}_7\left({\tau^\circ}^3 7^{3/2}\right) \ge \frac{878684768105600 \left(450888949-70747200 \sqrt{5}\right)}{93387562322467096652199\er}\approx 1.01312>1,
\]
thus proving the statement for all $d\ge 7$. Explicitly evaluating, using \eqref{eq:Fd},
\[
\mathfrak{F}_4\left({\tau^\circ}^3 4^{3/2}\right)\approx 1.07694,\quad  \mathfrak{F}_5\left({\tau^\circ}^3 5^{3/2}\right)\approx 1.1817, \quad \mathfrak{F}_6\left({\tau^\circ}^3 6^{3/2}\right)\approx 1.24933
\]
completes the proof.
\end{proof}

\subsection{An alternative test function}

We continue with another test function
\begin{equation}\label{eq:rhoNF}
\rho(x):=x^{\frac{1-d}{2}}\left(x-1+\frac{2}{d}\right)_+^{3/2}.
\end{equation}
Then, by explicit calculation,
\begin{equation}\label{eq:INF}
\begin{aligned}
\mathcal{I}_{d,1} &=\frac{-6 d^4+42 d^3-98 d^2+88 d-8+3 d (d-2)^2 \left(d^2-4 d+3\right) \ell_d}{4 d^3},\\
\mathcal{I}_{d,2} &=\frac{-6 d^2+18 d -8 +3 d (d-2)^2 \ell_d}{d^3},\\
\mathcal{I}_{d,3} &=\frac{4}{d^4},
\end{aligned}
\end{equation}
where we have denoted
\begin{equation}\label{eq:epsL}
\epsilon:=\frac{2}{d}\in(0,1),\qquad \ell_d:=\log\frac{d}{d-2}=-\log(1-\epsilon).
\end{equation}
Note that in this case $\mathfrak{b}_d>0$.

The experiments,  for this test function, are plotted in Figure \ref{fig:varNFMy}.

\begin{figure}[htb]
\centering
\includegraphics{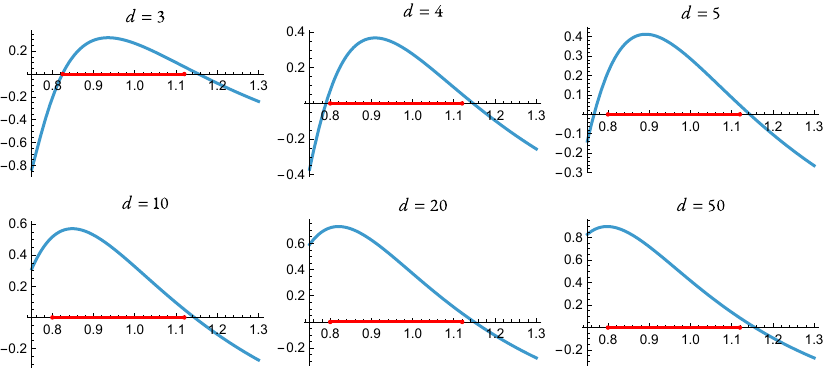}
\caption{Numerically evaluated left-hand side of \eqref{eq:N1staim} for the test function \eqref{eq:rhoNF}, shown as a function of $\tau$. The red lines show the positivity intervals guaranteed by Proposition \ref{prop:othertest}.  \label{fig:varNFMy}}
\end{figure}

\begin{lemma}\label{lem:piecewise-unimodal}
Let $d\ge 3$ and $\rho$ be given by \eqref{eq:rhoNF}. Then the function $\mathfrak{F}_d(\lambda)$ is unimodal in the variable $\lambda\in\left[\sqrt{\frac{\mathfrak{b}_d}{\mathfrak{a}_d}},\infty\right)$: that is,  it 
has exactly one critical point in that interval,  is strictly increasing to the
left of this critical point and strictly decreasing to the right. 
\end{lemma}

For the proof, see \S\ref{proof:testNF}.

\begin{prop}\label{prop:othertest}
With the choice of test function as above, for every $d\ge 3$ there exists an interval $\left[\tau_{d,-}^\bullet, \tau_{d,+}^\bullet\right]$ such that $\mathfrak{F}_d\left(\tau^3 d^{3/2}\right)-1>0$ for $\tau$ in this interval. Moreover, 
for all $d\ge 4$,  $\left[\tau_{d,-}^\bullet, \tau_{d,+}^\bullet\right]\supset [\tau_-^\bullet,\tau_+^\bullet] := \left[\frac{4}{5},\frac{28}{25}\right]$, and for $d=3$, $\left[\tau_{3,-}^\bullet, \tau_{3,+}^\bullet\right]\supset \left[\frac{33}{40},\frac{28}{25}\right]$.
\end{prop}

The case $d=3$ is dealt with by directly evaluating, using \eqref{eq:Fd},
\[
\mathfrak{F}_{3}\left(\left(\frac{33}{40}\right)^3 3^{3/2}\right)\approx 1.00521 >1, \qquad
\mathfrak{F}_{3}\left(\left(\frac{28}{25}\right)^3 3^{3/2}\right)\approx 1.0633 >1.
\]
The rest of Proposition \ref{prop:othertest} follows immediately from Lemma \ref{lem:piecewise-unimodal} and the following
\begin{prop}\label{prop:othertestvalues} 
Let $d\ge 4$. Then  $\mathfrak{F}_d\left({\tau_-^\bullet}^3 d^{3/2}\right)>1$ and  $\mathfrak{F}_d\left({\tau_+^\bullet}^3 d^{3/2}\right)>1$.
\end{prop}
For the proof, see \S\ref{proof:prop:othertestvalues}.

We summarise the results of Propositions \ref{prop:verylow}, \ref{prop:constant1}, \ref{prop:constant2}, and \ref{prop:othertest} in Figure \ref{fig:totallowpict}.

\begin{figure}[htb]
\centering
\includegraphics{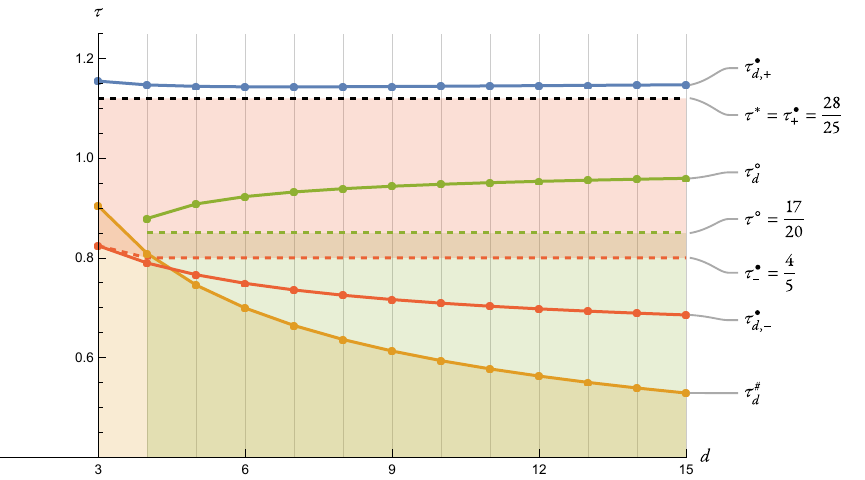}
\caption{Summary of the results of \S\ref{sec:lowlambda} in $(d,\tau)$-parameter space.\label{fig:totallowpict}}
\end{figure}

\section{Filling the gap}\label{sec:gap}

Whenever a finite gap was left in the analytic proofs in our previous papers \cite{FLPS, FLPS-AB, FLPS-Ann}, we have filled it using verified rational approximations of trapezoidal floor sums. Here, we use a different approach: namely, we create an algorithm to resolve

\

\noindent\textbf{Main task} 

\noindent\textit{Input}: an integer $d\ge 3$, a rational $\Lambda>0$ and a rational tolerance $\epsilon>0$. 

\noindent\textit{Output}: the \emph{complete} list $\mathtt{ListP}_{d,\Lambda}$ of all triples $\left(m,k,\overline{p'_{d,m,k}}\right)$ which includes \emph{all} pairs $(m,k)$ with $p'_{d,m,k}\le\Lambda$, 
with each entry also including a rational number $\overline{p'_{d,m,k}}$ such that
\[
p'_{d,m,k} < \overline{p'_{d,m,k}} \le \min\left\{p'_{d,m,k}+\epsilon, \Lambda\right\}
\]
(for the exceptional zero $p'_{d,0,1}=0$, the  output  is $(0,1,0)$).

The algorithm is designed in such a way that every claim is certified: all arithmetic is exact in $\mathbb{Q}$, all inequalities are proved, and no assertion relies on floating-point evaluation. 
Two-sided enclosures $\left[\underline{p'_{d,m,k}},\overline{p'_{d,m,k}}\right]\ni p'_{d,m,k}$ of width at most $\epsilon$ are in fact produced along the way.

\begin{definition}[Representations via a non-zero factor times a function $\mathbb{Q}\to\mathbb{Q}$]\label{defn:ratn}
Set, for $\nu\in \frac{1}{2}\mathbb{Z}$, $\nu\ge 0$, 
\begin{equation}\label{eq:deff}
F_\nu(t):=\Gamma(\nu+1)\left(\frac x2\right)^{-\nu}\left.J_\nu(x)\right|_{t=x^2/4}
=\sum_{j=0}^{\infty}\frac{(-t)^j}{j!\,\Pi_j(\nu)} ,
\end{equation}
which is entire in $t$; the series representation follows from the standard expansion \cite[Eq.~10.2.2]{dlmf}. Since $\nu\in\tfrac12\mathbb{Z}$, all coefficients of the series are rational.
If $f_{\nu, k}$ denotes the $k$th positive zero of $F_\nu$, then $f_{\nu,k}=\frac{j_{\nu,k}^2}{4}$.

Set also
\begin{equation}\label{eq:defh}
H_{d,m}(t):=m F_{m+\frac{d}{2}-1}(t)-\frac{2t}{m+\frac{d}{2}}\,F_{m+\frac{d}{2}}(t),
\end{equation}
which is entire with rational Taylor coefficients and $H_{d,m}(0)=m$. Then 
\begin{equation}\label{eq:uprime}
P_{d,m}'(x)
=\frac{x^{-\frac{d}{2}}(x/2)^{m+\frac{d}{2}-1}}{\Gamma(m+\frac{d}{2})}\; H_{d,m}\left(\frac{x^2}4\right).
\end{equation}
If $h_{d, m, k}$ denotes the $k$th positive zero of $H_{d,m}$, except that $h_{d, 0, 1}:=0$,  then $h_{d,m,k}=\frac{{p'_{d,m,k}}^2}{4}$.
\end{definition}

Our approach  follows the philosophy of computer-assisted proofs in spectral theory as surveyed in \cite{GS19} (see also the monographs \cite{Tuc11, NPW19}), and in particular the two-stage pattern (fast localisation, then rigorous certification) used in \cite{DS20,DGH21,DGP26}. It is based on Lemma \ref{lem:pandj} and the fact that the functions from Definition \ref{defn:ratn} can be evaluated at rational points with arbitrary certified precision due to  
\begin{lemma}\label{lem:enclosures}
Let $N\in\mathbb{N}_0$ and $t\in \mathbb{Q}$, $t>0$. Under notation of Definition \ref{defn:ratn}, define, additionally, the partial sums
\[
F_{\nu, N}(t):=\sum_{j=0}^{N}\frac{(-t)^j}{j!\Pi_j(\nu)}\in\mathbb{Q}, \qquad H_{d,m,N}(t):=m F_{m+\frac{d}{2}-1, N}(t)-\frac{2t}{m+\frac{d}{2}}F_{m+\frac{d}{2},N}(t)\in\mathbb{Q}.
\]
If $N$ is large enough so that 
\[
r_{\nu,N}(t):=\frac{t}{(N+2)\left(\nu+N+2\right)}<1,
\]
then
\[
F_\nu(t)\in\left[F_{\nu,N}(t) - \delta_{F_{\nu, N}}(t),F_{\nu,N}(t) + \delta_{F_{\nu, N}}(t)\right],
\]
with 
\[ 
\delta_{F_{\nu, N}}(t):=\frac{t^{N+1}}{(1-r_{\nu,N}(t))(N+1)!\Pi_{N+1}(\nu)}\in\mathbb{Q}.
\]
Similarly, if 
\[
r_{m+\frac{d}{2}-1,N}(t)<1,
\]
then
\[
H_{d,m}(t)\in\left[H_{d,m,N}(t) - \delta_{H_{d,m, N}}(t), H_{d,m,N}(t) + \delta_{H_{d,m, N}}(t)\right],
\]
with
\[
\delta_{H_{d,m, N}}(t):=m \delta_{F_{m+\frac{d}{2}-1,N}}(t)+\frac{2t}{m+\frac{d}{2}}\delta_{F_{m+\frac{d}{2},N}}(t)\in\mathbb{Q}.
\]
The remainder estimates tend to zero super-polynomially as $N\to\infty$.
\end{lemma}

\begin{proof} These are the standard remainder estimates of alternating series.
\end{proof}

This means that the roots of these functions can be also localised with arbitrary precision. For further details of the process and its theoretical justification we refer to Appendix \ref{app:rational}.

For each $d\in \{3,\dots,12\}$, we apply the procedure $\mathtt{MAIN}(d,\Lambda_{d}^*,\epsilon)$ with $\Lambda_{d}^*:=\ceiling{{\tau_d^*}^3 d^{3/2}}$ and $\epsilon=\frac{1}{100}$. Let $K_d=\#\mathtt{ListP}_{d,\Lambda_d^*}$. We re-arrange and re-index the list $\mathtt{ListP}_{d,\Lambda_d^*}$ into another list of quadruples $\left(n_i, m_i, k_i, \overline{p'_{d,m_i,k_i}}\right)$, $i=1,\dots, K_d$, with $\overline{p'_{d,m_i,k_i}} < \overline{p'_{d,m_{i+1},k_{i+1}}}$, $\left(n_1, m_1, k_1, \overline{p'_{d,m_1,k_1}}\right):=(1,0,1,0)$, and $n_{i+1}:=n_i+\kappa_{d,m_i}$, so that the Neumann eigenvalues of $\mathbb{B}^d$ satisfy
\[
\mu_{n_i}=\dots=\mu_{n_i+\kappa_{d,m_i}-1}<\left(\overline{p'_{d,m_i,k_i}}\right)^2.
\]
It  then remains to check if 
\begin{equation}\label{eq:checkpolya}
\max_{i=2,\dots,K_d} \frac{w_d \left(\overline{p'_{d,m_i,k_i}}\right)^d}{n_i-1}<1;
\end{equation}
all the numbers involved are rational except that for odd $d$ a factor $\frac{1}{\pi}$ appears in $w_d$ which we replace  by $\frac{1}{3}$. The summary of results is below; the last column shows a numerical approximation of the left-hand side of \eqref{eq:checkpolya}: the numbers are in fact exact but too long to display.

{\centering
\begin{longtable}{@{}*{5}{>{$}r<{$}}@{}}
\toprule
d & \Lambda_d^* & K_d & \mathcal{N}^\Neu_{\mathbb{B}^d}\left(\Lambda_d^*\right)&\text{approx. LHS of \eqref{eq:checkpolya}}\\
\midrule
 3 & 19 & 50 & 570 & 0.9061 \\
 4 & 22 & 63 & 4701 & 0.8533 \\
 5 & 25 & 77 & 36853 & 0.8396 \\
 6 & 29 & 101 & 379315 & 0.7688 \\
 7 & 34 & 133 & 4307698 & 0.8230 \\
 8 & 39 & 171 & 49018662 & 0.7520 \\
 9 & 44 & 215 & 650176935 & 0.7673 \\
 10 & 49 & 262 & 7562155655 & 0.7290 \\
 11 & 54 & 314 & 105849911256 & 0.7484 \\
 12 & 60 & 386 & 1524056239175 & 0.7016 \\
\bottomrule\\
\caption{Computer-assisted data}\label{table:data}
\end{longtable}
}

\begin{appendices} 
\renewcommand{\thesection}{\Alph{section}}

\section{Properties of $G_\lambda$}\label{app:G}

We recall that $G_\lambda$ is defined by \eqref{eq:Glambda}.

\begin{lemma}[{\cite[Lemma 4.5]{FLPS}}]\label{lem:Gdiff} 
The function $G_\lambda:[0,\lambda]\to\left[0,\frac{\lambda}{\pi}\right]$ is a strictly monotone decreasing convex $C^1$ function with
\[
\begin{alignedat}{3}
&&\qquad G_\lambda(0) &= \frac{\lambda}\pi,&\qquad G_\lambda(\lambda) &= 0,\\
G'_\lambda(z) &= - \frac{1}{\pi}\arccos \frac{z}{\lambda}, &\qquad G'_\lambda(0) &= - \frac{1}{2}, &\qquad G'_\lambda(\lambda) &= 0.
\end{alignedat}
\]
\end{lemma}

\begin{lemma}[{\cite[Corollary 4.7]{FLPS}}]\label{lem:GintW} 
Let $l\in\mathbb{N}_0$. Then
\[
\frac{2}{l!} \int_0^{\lambda} z^{l} G_\lambda(z)\, \dr z = w_{l+2}\lambda^{l+2}.
\]
\end{lemma}

The following result improves \cite[Lemma 6.2]{FLPS-Ann}. Set
\begin{equation}\label{eq:c1c2}
c_1:=\frac{2\sqrt{2}}{3\pi}\approx 0.300105, \qquad c_2:=\frac{1}{5}\left(1-\frac{2\sqrt{2}}{\pi}\right)\approx 0.0199367.
\end{equation}

\begin{lemma}\label{lem:Gbounds}
Let $0<z<\lambda$, and let
\[
\theta:=1-\frac{z}{\lambda}.
\]
Then
\[
c_1\lambda \theta^{3/2} < G_\lambda(z) <
\lambda\left(c_1\theta^{3/2} + c_2\theta^{5/2}\right).
\]
\end{lemma}

\begin{proof}
We first prove that, for $0<w=v\lambda<\lambda$, with $0<v<1$, 
we have
\begin{equation}\label{eq:wineq-improved}
\frac{\sqrt{2}}{\pi}\sqrt{1-v} < -G'_\lambda(w) = \frac{1}{\pi}\arccos v 
< \left(\frac{1}{2} -
\left(\frac{1}{2}-\frac{\sqrt{2}}{\pi}\right) v \right) \sqrt{1-v}.
\end{equation}

The left inequality in \eqref{eq:wineq-improved} is equivalent to 
\begin{equation}\label{eq:vineq-improved}
\cos\left(\sqrt2\sqrt{1-v}\right) > v.
\end{equation}
We have
\[
\left. \left(
\cos\left(\sqrt{2}\sqrt{1-v}\right)-v
\right) \right|_{v=1}=0,
\]
and
\[
\frac{\dr}{\dr v}
\left(
\cos\left(\sqrt2\sqrt{1-v}\right)-v
\right) =\frac{\sin\left(\sqrt{2}\sqrt{1-v}\right)}{\sqrt{2}\sqrt{1-v}}-1
<0\qquad\text{for all } v\in(0,1),
\]
hence \eqref{eq:vineq-improved} holds for all $v\in(0,1)$.

In order to prove the right inequality in \eqref{eq:wineq-improved}, we introduce yet another variable
\[
t:=\sqrt{1-v}\in(0,1),
\]
and restate the inequality, after minimal simplifications, as
\begin{equation}\label{eq:phiineq}
h(t):=\sqrt{2}t+ c t^3-\arccos\left(1-t^2\right)>0
\qquad \text{for all }0< t <1,
\end{equation}
with
\[ 
c:=\frac{\pi}{2}-\sqrt{2}\approx 0.156583.
\]
It is easy to check that
\[
h(0)=h'(0)=h''(0)=h(1)=0,\qquad h'''(0)>0,\qquad h'(1), h''(1)<0.
\]
Moreover, using the definition of the constant $c$,  we obtain that
\[
h''(t)
=
t\left(
3\pi-6\sqrt2-\frac{2}{(2-t^2)^{3/2}}
\right)
\]
has a single zero in the interval $(0,1)$, see Figure \ref{fig:phi}, and thus proves \eqref{eq:phiineq} and therefore the right inequality in \eqref{eq:wineq-improved}.

\begin{figure}[htb]
\centering
\includegraphics{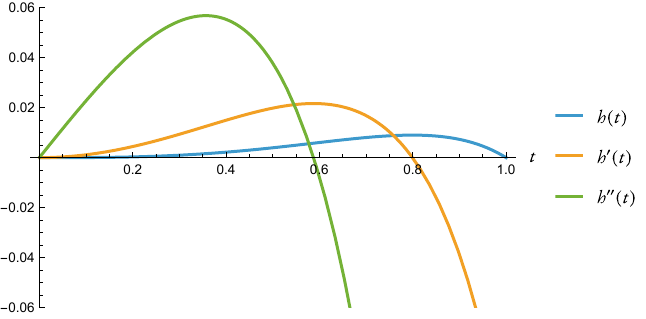}
\caption{Function $h(t)$ and its derivatives.\label{fig:phi}}
\end{figure}

Finally, since $G_\lambda(\lambda)=0$, we have
\[
G_\lambda(z)
=
\int_z^\lambda \left(-G'_\lambda(w)\right)\,\dr w.
\]
Integrating the left inequality \eqref{eq:wineq-improved} from $z$ to $\lambda$ gives the lower bound
\[
G_\lambda(z)
>
\frac{\sqrt2}{\pi}
\int_z^\lambda
\left(1-\frac{w}{\lambda}\right)^{1/2}\,\dr w
=
c_1\lambda
\left(1-\frac{z}{\lambda}\right)^{3/2}.
\]
For the upper bound, we obtain
\[
G_\lambda(z)
<
\int_z^\lambda
\left(
\frac{1}{2}
-
\left(\frac{1}{2}-\frac{\sqrt2}{\pi}\right)\frac{w}{\lambda}
\right)
\left(1-\frac{w}{\lambda}\right)^{1/2}
\,\dr w.
\]
Writing $u=1-w/\lambda$, this becomes
\[
G_\lambda(z)
<
\lambda
\int_0^\theta
\left(
\frac{\sqrt2}{\pi}
+
\left(\frac{1}{2}-\frac{\sqrt2}{\pi}\right)u
\right)
u^{1/2}\,\dr u,
\]
therefore
\[
G_\lambda(z)
<
\lambda\left(
c_1\theta^{3/2}
+
c_2\theta^{5/2}
\right),
\]
as claimed.
\end{proof}

\section{Combinatorial bounds}\label{app:comb}
Throughout this appendix,  $\kappa_{d,m}$ is defined by \eqref{eq:kappadm}, and 
\begin{equation}\label{eq:Pin}
\Pi_n(x):=(x+1)\times\cdots\times(x+n), \qquad x\ge -1, n\in\mathbb{N}; \quad \Pi_0(x):=1 
\end{equation}
is related to Pochhammer's symbol or rising factorial $(x)_n:=x(x+1)\times \cdots \times  (x+n-1)$ by $\Pi_n(x)=(x+1)_n$.

\begin{lemma}\label{lem:kappas}
Let $d\ge 3$, $m\ge 0$, and $l\ge 0$. Then
\begin{equation}\label{eq:kappadmalt}
\kappa_{d,m}=\frac{2m+d-2}{(d-2)!}\Pi_{d-3}(m)
\end{equation}
and
\begin{equation}\label{eq:kappadmsum}
\sum_{j=0}^{l} \kappa_{d,j}=\frac{2\left(l+\frac{d-1}{2}\right)}{(d-1)!}\Pi_{d-2}(l).
\end{equation}
\end{lemma}

\begin{proof} 
These identities are standard in the theory of spherical harmonics, see, e.g. \cite{AtHa}. They are proved by a direct computation, and we include the proofs for completeness.  

We have
\[
\begin{split}
\kappa_{d,m}&=\binom{m+d-1}{d-1}-\binom{m+d-3}{d-1}=\frac{1}{(d-1)!}\left(\Pi_{d-1}(m)-\Pi_{d-1}(m-2)\right)\\
&=\frac{1}{(d-1)!}\Pi_{d-3}(m)\left((m+d-2)(m+d-1)-(m-1)m\right)=\frac{2m+d-2}{(d-2)!}\Pi_{d-3}(m),
\end{split}
\]
and
\[
\begin{split}
\sum_{k=0}^{l} \kappa_{d,k}&=
\sum_{k=0}^{l} \left(\binom{k+d-1}{d-1}-\binom{k+d-3}{d-1}\right)\\
&=\binom{l+d}{d}-\binom{l+d-2}{d}=\binom{l+d-1}{d-1}+\binom{l+d-1}{d}-\binom{l+d-2}{d}\\
&=\binom{l+d-2}{d-1}+\binom{l+d-1}{d-1}
=\frac{2\left(l+\frac{d-1}{2}\right)\Pi_{d-2}(l)}{(d-1)!},
\end{split}
\]
where we have additionally used the standard identity \cite[Eq. 26.3.7]{dlmf} 
\[
\sum_{i=l}^{r}\binom{i}{l}=\binom{r+1}{l+1},\qquad r\ge l,
\]
and another identity \cite[Eq. 26.3.5]{dlmf},
\[
\binom{k}{l}=\binom{k+1}{l+1}-\binom{k}{l+1},\qquad k\ge l.
\]
\end{proof}

\begin{lemma}\label{lem:Pi1}
Let $x\ge -1$, $n\in\mathbb{N}$. Then
\begin{equation}\label{eq:Pibounds1}
\left(x+\frac{n+1}{2}\right)^n -\alpha_n\left(x+\frac{n+1}{2}\right)^{n-2}\le \Pi_n(x)\le \left(x+\frac{n+1}{2}\right)^n,
\end{equation}
where
\[
\alpha_n:=\frac{(n-1)n(n+1)}{24}.
\]
\end{lemma}

\begin{proof}
Set
\[
y:=x+\frac{n+1}{2}.
\]
Then $y\ge (n-1)/2$.

To prove the upper bound, we apply the AM-GM inequality
\[
l\left(\beta_1\cdots \beta_l\right)^{1/l}\le \beta_1+\dots+\beta_l,\qquad l\in\mathbb{N},\quad \beta_1,\dots,\beta_l\ge 0, 
\]
with $l=n$ and $\beta_j=x+j$, $j=1,\dots,n$, and raise both sides of the result to the power $n$.
Since
\[
\frac1n\sum_{j=1}^n(x+j)=x+\frac{n+1}{2}=y,
\]
we obtain
\[
\Pi_n(x)=\prod_{j=1}^n(x+j)\le y^n.
\]

For the lower bound, we distinguish the parity of $n$.

If $n=2m$, then
\[
\Pi_n(x)
=\prod_{k=1}^m\left(y-(k-\frac12)\right)\left(y+(k-\frac12)\right)
= y^n \prod_{k=1}^m\left(1-\frac{(k-\frac12)^2}{y^2}\right).
\]
Since $y\ge m-\frac12$, all factors on the right belong to $[0,1]$. Hence, using
\[
\prod_{j=1}^m (1-u_j)\ge 1-\sum_{j=1}^m u_j,
\qquad 0\le u_j\le 1,
\]
we get
\[
\Pi_n(x)\ge y^n-\left(\sum_{k=1}^m\left(k-\frac12\right)^2\right)y^{n-2}.
\]
Now
\[
\sum_{k=1}^m\left(k-\frac12\right)^2
=\frac{m(4m^2-1)}{12}
=\frac{(n-1)n(n+1)}{24}
=\alpha_n.
\]

If $n=2m+1$, then
\[
\Pi_n(x)
= y\prod_{k=1}^m (y-k)(y+k)
= y^n \prod_{k=1}^m\left(1-\frac{k^2}{y^2}\right).
\]
Again $y\ge m$, so all factors belong to $[0,1]$, and therefore
\[
\Pi_n(x)\ge y^n-\left(\sum_{k=1}^m k^2\right)y^{n-2}.
\]
Since
\[
\sum_{k=1}^m k^2
=\frac{m(m+1)(2m+1)}{6}
=\frac{(n-1)n(n+1)}{24}
=\alpha_n,
\]
the lower bound follows also in this case.

This proves \eqref{eq:Pibounds1}.
\end{proof}

\begin{lemma}\label{lem:Pi3}
Let $x\ge -1$, $n\in\mathbb{N}$. Then
\begin{equation}\label{eq:Pibounds2}
\left(x+\frac{n+1}{2}\right)\Pi_n(x)\ge \left(x+\frac{n}{2}+1\right)^{n+1} -\frac{n+1}{2}\left(x+\frac{n}{2}+1\right)^n -\gamma_n\left(x+\frac{n}{2}+1\right)^{n-1}=S_{n+2}\left(x+\frac{n}{2}+1\right),
\end{equation}
where 
\begin{equation}\label{eq:gamman}
\gamma_n:=\frac{(n-4)n(n+1)}{24},
\end{equation}
and $S_{n+2}$ is defined by \eqref{eq:Sd}.
\end{lemma}

\begin{proof}
Set
\[
y:=x+\frac{n+1}{2}\ge 0,\qquad z:=x+\frac{n}{2}+1=y+\frac{1}{2}.
\]
By Lemma~\ref{lem:Pi1},
\[
y\,\Pi_n(x)\ge y^{n+1}-\alpha_n y^{n-1},
\]
so it suffices to show that for $y\ge 0$ and $n\in\mathbb N$,
\begin{equation}\label{eq:sufficesforPi3}
y^{n+1}-\alpha_n y^{n-1}\ge \left(y+\frac 12\right)^{n-1}\left(y^2-\frac{n-1}{2}y-\frac{n}{4}-\gamma_n\right).
\end{equation}
Expanding the right-hand side of \eqref{eq:sufficesforPi3} using a binomial series, and rearranging terms, it becomes
\begin{equation}\label{eq:suffforsuff}
\sum_{j=2}^{n+1}\binom{n-1}{n+1-j}\left(\frac 12\right)^{n+1-j} y^j -\frac{n-1}{2}\sum_{j=1}^{n}\binom{n-1}{n-j}\left(\frac 12\right)^{n-j}y^j - \left(\frac n4+\gamma_n\right)\sum_{j=0}^{n-1}\binom{n-1}{n-1-j}\left(\frac 12\right)^{n-1-j}y^j.
\end{equation}
The coefficient of $y^{n+1}$ is equal to one, the coefficient of $y^n$ is zero, and the coefficient of $y^{n-1}$ is
\[
-\frac{(n-1)n}{8}-\frac{n}{4}-\gamma_n = -\alpha_n.
\]
All of these are equal to the coefficients of the left-hand side of \eqref{eq:sufficesforPi3}. It remains to show that all coefficients of $y^j$, $j=0,\dots,n-2$, of \eqref{eq:suffforsuff} are negative. This may easily be checked directly for $j=0$ and $j=1$. For $2\le j\le n-2$, the coefficient is, by symmetry of binomial coefficients,
\[
\left(\frac 12\right)^{n-1-j}\left(\frac 14\binom{n-1}{j-2} - \frac{n-1}{4}\binom{n-1}{j-1}-\left(\frac n4+\gamma_n\right)\binom{n-1}{j}\right).
\]
Factoring out $\binom{n-1}{j}$, this is
\[
\left(\frac 12\right)^{n-1-j}\binom{n-1}{j}\left(\frac 14\frac{(j-1)j}{(n-j)(n-(j-1))} - \frac{n-1}{4}\frac{j}{n-j}-\frac{1}{24}n(n-1)(n-2))\right).
\]
Simplifying this yields
\[
\left(\frac 12\right)^{n-1-j}\binom{n-1}{j}\frac{1}{4(n-j)(n-j+1)}\left(-nj(n-j) - \frac 16n(n-1)(n-2)(n-j)(n-j+1) \right),
\]
which is negative, as desired.
\end{proof}

\begin{lemma}\label{lem:Gammaratio1}
For every $d\in\mathbb{N}$, and  with $\mathcal{R}_d$ given by  \eqref{eq:Rd},
\begin{equation}\label{eq:Rdbounds}
1+\frac{1}{4d} < \mathcal{R}_d  < 1+\frac{1}{4d}+\frac{1}{32d^2}.
\end{equation}
\end{lemma}

\begin{proof}  Set
\[
r_d:=\frac{\mathcal{R}_d}{1+\frac{1}{4d}+\frac{1}{32d^2}},\qquad \tilde{r}_d:=\frac{\mathcal{R}_d}{1+\frac{1}{4d}};
\]
our aim is to show that $r_d<1<\tilde{r}_d$ for all $d$. 

By Stirling's logarithmic expansion,
\[
\mathcal{R}_d = 1+\frac{1}{4d}+\frac{1}{32d^2} + O\left(d^{-3}\right)\qquad\text{as }d\to+\infty,
\]
and therefore 
\[
\lim_{d\to\infty}\mathcal{R}_d = \lim_{d\to\infty}r_d = \lim_{d\to\infty}\tilde{r}_d=1.
\]

Explicitly, for every $d\ge 3$, 
\[
r_d^2-r_{d-2}^2 = \frac{2048 \left(480 d^5-1672 d^4+1095 d^3+614 d^2+100 d+8\right) \Gamma \left(\frac{d}{2}\right)^2}{(d-1)^2 \left(32 d^2-120 d+113\right)^2 \left(32 d^2+8 d+1\right)^2 \Gamma \left(\frac{d-1}{2}\right)^2}>0,
\]
with the positivity following from the fact that the quintic polynomial in the numerator is bounded below by $d^3 \left(480 d^2-1672 d+1095\right) > 0$. Hence the two subsequences $\left\{r_{2k-j}\right\}_{k=1}^\infty$, $j=0,1$,  are monotone increasing towards their joint limit one. Also,
\[
\tilde{r}_{d-2}^{\,2}-\tilde{r}_d^{\,2} =\frac{8 \left(4 d^2-11 d-2\right) \Gamma \left(\frac{d}{2}\right)^2}{(16 d^2-24 d-7)^2 \Gamma \left(\frac{d+1}{2}\right)^2}>0,
\] 
hence the two subsequences $\left\{\tilde{r}_{2k-j}\right\}_{k=1}^\infty$, $j=0,1$,  are monotone decreasing towards their joint limit one, which completes the proof.
\end{proof}

Recall that
\[
w_d:=\frac{1}{2^d\Gamma\left(\frac d2+1\right)^2},
\qquad d\ge 3.
\]

\begin{lemma}\label{lem:wdrd}
Let $d\ge 3$. 
Then
\begin{alignat}{2}
\frac{w_{d-1}}{w_d} &= d\mathcal{R}_d^2,&\qquad \frac{w_{d-2}}{w_d} &=d^2,\label{eq:wratios}\\
\frac{1}{(d-1)!\,w_d} &= \sqrt{\frac{\pi}{2}}d^{3/2}\mathcal{R}_d,&\qquad \frac{1}{(d-2)!\,w_d} &= \sqrt{\frac{\pi}{2}}(d-1)d^{3/2}\mathcal{R}_d.\label{eq:factwratios}
\end{alignat}
\end{lemma}

\begin{proof}
By the duplication formula,
\[
\Gamma\left(\frac d2+1\right)\Gamma\left(\frac d2+\frac{1}{2}\right)
=
2^{-d}\sqrt{\pi}\,\Gamma(d+1)
=
2^{-d}\sqrt{\pi}\,d!,
\]
and relations \eqref{eq:wratios}, \eqref{eq:factwratios} follow by elementary transformations.
\end{proof}

\section{Proofs of some technical results}\label{app:technical} 
\subsection{Proof of Lemma \ref{lem:prodint}}\label{proof:lem:prodint}
We have 
\[
\utilde{F}(z)-F(z)= \int_{0}^z \left(\utilde{F}'(t)-f(t)\right)\,\dr t\le 0.
\]
Integrating by parts, we get
\[
\begin{split}
 \int_{0}^y \utilde{F}'(z) g(z)\,\dr z -  \int_0^y f(z)g(z)\,\dr z &= \int_0^y \left(\utilde{F}'(z)-F'(z)\right) g(z) \,\dr z\\
&= \left.\left(\utilde{F}(z)-F(z)\right)g(z)\right|_0^y-\int_0^y \left(\utilde{F}(z)-F(z)\right)g'(z) \,\dr z\le 0
\end{split}
\]
since $\utilde{F}(0)=F(0)=0$, $\utilde{F}(y) \le F(y)$, $g(y)\ge 0$, and $(\utilde{F}(z)-F(z))g'(z)\ge 0$, thus proving  \eqref{eq:fgprod}.

\subsection{Proof of Proposition \ref{prop:Q1b} (the monotonicity of $\mathcal{T}_{d,1}(\tau)$ in $d$ part)}\label{proof:prop:Q1b}
First of all, we observe that both coefficients $t_{1,0}(d)$ and $t_{1,3}(d)$  are non-negative, with the former monotone decreasing in $d$ and the latter monotone increasing. 

For $3\le d\le 6$, one has $(d-6)_+=0$, and therefore $\mathcal{T}_{d,1}(\tau)=t_{1,0}(d)$  is strictly decreasing in $d$. 

It remains to consider $d\ge 6$, which we now regard as a real variable. Then
\[
\frac{\partial}{\partial d}\mathcal{T}_{d,1}(\tau)
= - \frac{1536 d^4-1600 d^{7/2}\tau^{-3}+1920 d^3+912 d^2+192 d+15}{384 d^{7/2} (4 d+1)^2},
\]
which is obviously negative for $\tau\ge 1$ and $d\ge 6$ since $1536 d^4-1600 d^{7/2}\tau^{-3} \ge 64 d^{7/2}\left(24\sqrt{d}-25\right)$ is positive.
It follows that $\mathcal{T}_{d,1}(\tau)$ is decreasing as a function of the real variable $d$ on $[6,\infty)$, and hence also decreasing along the integers $d\ge 6$.
Combining this with the monotonicity for $3\le d\le 6$, we obtain the claimed monotone decrease in $d$.

\subsection{Proof of Proposition \ref{prop:Q2b} (the monotonicity of $\mathcal{T}_{d,2}(\tau)$ in $\tau$ and $d$ part)}\label{proof:prop:Q2b}

Throughout this  proof we regard $d\ge 3$ as a real variable, and  set
\[
\sigma:=\frac{1}{\tau^{2}},\qquad \xi:=\frac{1}{d}.
\]
We will prove that the function
\[
\mathcal{E}(\sigma,\xi):=\frac{\mathcal{T}_{1/\xi,2}\left(\frac{1}{\sqrt{\sigma}}\right)}{t_{2,2}}
\]
is monotone increasing in both variables in the region $(\sigma,\xi)\in(0,1)\times\left(0,\frac{1}{3}\right]$.

By \eqref{eq:T2}, we have
\begin{equation}\label{eq:E}
\begin{split}
\mathcal{E}(\sigma,\xi)&=\sigma+\alpha_{4}\left(\frac{1}{\xi}\right)\sigma^2+\alpha_{6}\left(\frac{1}{\xi}\right)\sigma^3+\alpha_{8}\left(\frac{1}{\xi}\right)\sigma^4\\
&=\sigma + \left(-\beta_{4,0}+\beta_{4,1}\xi\right)\sigma^2 +  \left(\beta_{6,0}-\beta_{6,1}\xi+\beta_{6,2}\xi^2\right)\sigma^3 +  \left(\beta_{8,1}\xi-\beta_{8,2}\xi^2+\beta_{8,3}\xi^3\right)\sigma^4,
\end{split}
\end{equation}
where the coefficients
\begin{gather*}
\beta_{4,0}=\frac{5 (3 \pi )^{2/3}}{28\ 2^{1/3}}\approx 0.632,\qquad \beta_{4,1}=\frac{(3 \pi )^{2/3} \left(14 \sqrt{2}+3 \pi \right)}{56\ 2^{5/6}}\approx 1.307,\\
\beta_{6,0}=\frac{5 \pi ^{4/3}}{32\ 6^{2/3}} \approx  0.218,\qquad \beta_{6,1}=\frac{\pi ^{4/3} \left(3 \pi  \sqrt{2}+38\right)}{64\ 6^{2/3}} \approx 1.118,\qquad \beta_{6,2}=\frac{\sqrt[3]{3} \pi ^{4/3} \left(\pi  \sqrt{2}+6\right)}{32\ 2^{2/3}} \approx 1.364,\\
\beta_{8,1}=\frac{27 \pi ^2 \left(\sqrt{2} \pi -4\right)}{11264} \approx  0.0105,\ \ \beta_{8,2}=\frac{135 \pi ^2 \left(\sqrt{2} \pi -4\right)}{11264} \approx 0.0524,\ \ \beta_{8,3}=\frac{81 \pi ^2 \left(\sqrt{2} \pi -4\right)}{5632} \approx 0.0629,
\end{gather*}
are all chosen to be positive.

Differentiating \eqref{eq:E} with respect to $\sigma$, dropping the positive terms containing $\beta_{6,2}$, $\beta_{8,1}$, and  $\beta_{8,3}$, and estimating the negative term containing $-\beta_{8,2}\xi^2\sigma^3 \ge -\frac{1}{9}\beta_{8,2}$, we get
\[
\frac{\partial\mathcal{E}(\sigma, \xi)}{\partial\sigma}>\xi  \sigma  \left(2 \beta_{4,1}-3 \sigma  \beta_{6,1}\right)+1-\frac{4}{9} \beta_{8,2}-2 \sigma  \beta_{4,0}+3 \sigma ^2 \beta_{6,0}.
\] 
As the right-hand side is linear in $\xi$, it is sufficient to check its positivity only for $\xi\in\left\{0,\frac{1}{3}\right\}$. We have, by minimising the quadratic expression in $\sigma$, 
\[
\begin{split}
\left.\frac{\partial\mathcal{E}(\sigma, \xi)}{\partial\sigma}\right|_{\xi=0}&>1-\frac{4}{9} \beta_{8,2}-2 \sigma  \beta_{4,0}+3 \sigma ^2 \beta_{6,0}\ge -\frac{\beta_{4,0}^2}{3 \beta_{6,0}}+1-\frac{4}{9}\beta_{8,2}\\
&=\frac{19}{49}-\frac{15 \pi ^2 \left(\sqrt{2} \pi -4\right)}{2816}\approx 0.364>0,
\end{split}
\]
and also
\[
\begin{split}
\left.\frac{\partial\mathcal{E}(\sigma, \xi)}{\partial\sigma}\right|_{\xi=\frac{1}{3}}&>\sigma ^2 \left(3 \beta_{6,0}-\beta_{6,1}\right)+\sigma  \left(\frac{2 \beta_{4,1}}{3}-2 \beta_{4,0}\right)+1-\frac{4 \beta_{8,2}}{9}\\
&\ge  \left(3 \beta_{6,0}-\beta_{6,1}\right)+ \left(\frac{2 \beta_{4,1}}{3}-2 \beta_{4,0}\right)+1-\frac{4 \beta_{8,2}}{9}\approx 0.119 >0,
\end{split}
\]
where the second inequality is due to the fact that the coefficients of terms with $\sigma$ and $\sigma^2$ are both negative, and therefore these terms can be estimated from below by taking $\sigma=1$.
This finishes the proof of monotonicity in $\sigma$.

In a similar manner, differentiating \eqref{eq:E} with respect to $\xi$ and dropping the positive term containing $\beta_{8,3}$, we get
\[
\frac{\partial\mathcal{E}(\sigma, \xi)}{\sigma^2\partial\xi}>2 \xi  \sigma  \left(\beta_{6,2}-\sigma  \beta_{8,2}\right)+\sigma  \left(\sigma  \beta_{8,1}-\beta_{6,1}\right)+\beta_{4,1},
\]
and it is again sufficient to check that the right-hand side is positive for $\xi\in\left\{0,\frac{1}{3}\right\}$. We have
\[
\left.\frac{\partial\mathcal{E}(\sigma, \xi)}{\sigma^2\partial\xi}\right|_{\xi=0} > \sigma^2  \beta_{8,1} - \sigma \beta_{6,1}+\beta_{4,1}> \beta_{4,1}  - \beta_{6,1} \approx 0.189 >0,
\]
and 
\[
\begin{split}
\left.\frac{\partial\mathcal{E}(\sigma, \xi)}{\sigma^2\partial\xi}\right|_{\xi=\frac{1}{3}} &> \sigma ^2 \left(\beta_{8,1}-\frac{2 \beta_{8,2}}{3}\right)+\sigma  \left(\frac{2 \beta_{6,2}}{3}-\beta_{6,1}\right)+\beta_{4,1}\\
&>\beta_{4,1} -\beta_{6,1} - \frac{2 \beta_{8,2}}{3}\approx 0.154 >0,
\end{split}
\]
which completes the proof of monotonicity in $\xi$.

\subsection{Proof of Lemma \ref{lem:monder}}\label{proof:lem:monder}

First of all, $\omega<1$ by \eqref{eq:omega}, and therefore $\utilde{\mathcal{L}}'(\lambda)=1-\frac{\omega}{3}\lambda^{-2/3}>0$ for $\lambda\ge d^{3/2}$, so $\utilde{\mathcal{L}}$ is increasing on $\left[d^{3/2},\infty\right)$. Thus
\[
\utilde{\mathcal{L}}(\lambda)
\ge
\utilde{\mathcal{L}}\left(d^{3/2}\right)
=
d^{3/2}-\omega d^{1/2}
>
d^{3/2}-d^{1/2}
=
\sqrt{d} (d-1),
\]
and it is enough to establish positivity of 
\[
S_d'(z)
=
(d-1) z^{d-4}
\left(
z^2-\frac{d-2}{2}\,z-\frac{(d-3)(d-2)(d-6)}{24}
\right)
\]
for all $z\ge \sqrt{d}(d-1)$. As
\[
\frac{d-2}{2}<\frac{d-1}{2}\le \frac{\sqrt{d}\,(d-1)}{2}\le\frac z2,
\]
we have
\[
\begin{split}
z^2-\frac{d-2}{2}\,z-\frac{(d-3)(d-2)(d-6)}{24}
&>
\frac{z^2}{2}-\frac{(d-3)(d-2)(d-6)}{24}\\
&\ge
\frac{d(d-1)^2}{2}-\frac{(d-3)(d-2)(d-6)}{24},
\end{split}
\]
which is positive for all $d\ge 3$. 
Therefore $S_d'(z)>0$ for all $z\ge\utilde{\mathcal{L}}(\lambda)$, as claimed.

\subsection{A bound for $\log$}

\begin{lemma}\label{lem:log}
Let $x\in(0,1)$. Then 
\[
-x>\log(1-x)>-x-\frac{x^2}{2(1-x)}.
\]
\end{lemma}

\begin{proof}
We have 
\[
-x>\log(1-x)=-x-\sum_{k=2}^\infty \frac{x^k}{k} > -x-\sum_{k=2}^\infty \frac{x^k}{2} = -x-\frac{x^2}{2(1-x)},
\]
where the last equality is obtained by summing up the geometric series.
\end{proof}

As an immediate corollary, we obtain the lower bound
\begin{equation}\label{eq:loglower}
\log\left(1-\frac{2}{d}\right) > -2\frac{d-1}{d(d-2)}, 
\end{equation}
valid for all $d\ge 3$.

\subsection{Proof of Proposition \ref{prop:Q3b} (the monotonicity of $\mathcal{T}_{d,3}(\tau)$ in $\tau$ and $d$ part)}\label{proof:prop:Q3b}

By \eqref{eq:T3}, we may write
\[
\mathcal{T}_{d,3}(\tau)
=
\frac{1}{2}\sqrt{\frac{\pi}{2}}\,\er^{-\omega\tau^{-2}} \Psi(\tau, d),
\qquad
\Psi(\tau,d)
:=
1-\frac{\phi_1(d)}{\phi_3(\tau,d)}-\frac{\phi_2(d)}{\phi_3(\tau,d)^2}
\]
where
\[
\phi_3(\tau, d):=d\tau^3-\omega\tau,
\qquad
\phi_1(d):=\frac{d-1}{2\sqrt{d}},
\qquad
\phi_2(d):=\frac{(d-1)(d-2)(d-6)}{24 d}.
\]
The factor $\er^{-\omega\tau^{-2}}$ is positive, independent of $d$, and strictly increasing in $\tau$. It is therefore enough to prove that $\Psi(\tau, d)$ is positive and strictly increasing in each of the two variables $(\tau, d)\in[1,+\infty)\times[3,+\infty)$, treating $d$ as a real variable.

Since $0<\omega<1$ by \eqref{eq:omega}, we have, for $d\ge 3$ and $\tau\ge 1$,
\begin{equation}\label{eq:ellbounds}
\phi_3(\tau, d)\ge (d-\omega)\tau^3>(d-1)\tau^3\ge d-1\ge 2,
\qquad
\frac{\partial \phi_3(\tau, d)}{\partial \tau}=3d\tau^2-\omega>0,
\end{equation}
so $\phi_3(\tau,d)$ is positive and strictly increasing in $\tau$.

Differentiating in $\tau$,
\[
\frac{\partial}{\partial\tau}\Psi(\tau, d)
=
\left(\phi_1(d)+\frac{2\phi_2(d)}{\phi_3(\tau, d)}\right)
\frac{3d\tau^2-\omega}{\phi_3(\tau, d)^2},
\]
so by \eqref{eq:ellbounds} it suffices to check that the first factor is positive. For $d\ge 6$ we have $\phi_2(d)\ge 0$, so this is obviously the case. For $d\in\{3,4,5\}$, 
a direct check using \eqref{eq:ellbounds} shows that
\[
\phi_1(d)+\frac{2\phi_2(d)}{\phi_3(\tau, d)}\ge \frac{d-1}{2\sqrt{d}} - \frac{(d-1)(d-2)(6-d)}{24 d}=\frac{(d-1)}{2\sqrt{d}}\left(1-\frac{(d-2)(6-d)}{12\sqrt{d}}\right)>0.
\]

A direct differentiation in $d$ using $\frac{\partial\phi_3(\tau,d)}{\partial d}=\tau^3$ gives
\begin{equation}\label{eq:dAdB}
\begin{split}
\frac{\partial\Psi(\tau, d)}{\partial d}
&=
-\frac{\partial}{\partial d}\frac{\phi_1(d)}{\phi_3(\tau, d)}
-\frac{\partial}{\partial d}\frac{\phi_2(d)}{\phi_3(\tau, d)^2},\\
-\frac{\partial}{\partial d}\frac{\phi_1(d)}{\phi_3(\tau, d)}
&=
\frac{\tau\left(d(d-3)\tau^2+(d+1)\omega\right)}{4d^{3/2} \phi_3(\tau, d)^2},\\
\frac{\partial}{\partial d}\frac{\phi_2(d)}{\phi_3(\tau, d)^2}
&=
\frac{\tau\left(\left(9d-40+\frac{36}{d}\right)\tau^2-\left(2d-9+\frac{12}{d^2}\right)\omega\right)}{24\phi_3(\tau, d)^3}.
\end{split}
\end{equation}

Consider first the case $d\ge 6$. Then $2d-9>0$, and dropping the two $\omega$-terms in \eqref{eq:dAdB} (which decreases the first quantity and increases the second one) yields
\[
-\frac{\partial}{\partial d}\frac{\phi_1(d)}{\phi_3(\tau, d)}
>
\frac{(d-3)\tau^3}{4\sqrt{d}\,\phi_3(\tau, d)^2},
\qquad
\frac{\partial}{\partial d}\frac{\phi_2(d)}{\phi_3(\tau, d)^2}
\le
\frac{\left(9d^2-40d+36\right)\tau^3}{24d\phi_3(\tau, d)^3}
\le
\frac{3(d-2)^2\tau^3}{8d\phi_3(\tau, d)^3},
\]
where in the last step we used $9d^2-40d+36=9(d-2)^2-4d<9(d-2)^2$. Hence $\frac{\partial}{\partial d}\Psi(\tau, d)>0$ would follow from
\[
2\sqrt{d}\,(d-3)\phi_3(\tau, d)\ge 3(d-2)^2.
\]
Since $\phi_3(\tau, d)>d-1$ by \eqref{eq:ellbounds}, and $\sqrt{d}\ge \sqrt{6}$, we have
\[
2\sqrt{d} (d-3) \phi_3(\tau, d) - 3(d-2)^2 \ge 2\sqrt{6} (d-1)(d-3) - 3(d-2)^2 = (2\sqrt{6}-3)\left(d^2-4 d+\frac{2}{5} \left(6-\sqrt{6}\right)\right)>0
\]
as required.

It remains to consider $3\le d\le 6$, for which we argue directly from \eqref{eq:dAdB}. After rearrangement, the positivity of the derivative in question is equivalent to the following inequality:
\begin{equation}\label{eq:positivityinequality}
6(d\tau^3-\omega\tau)(d(d-3)\tau^2+(d+1)\omega)\ge \sqrt{d}(9d^2-40d+36)\tau^2-\sqrt{d}(2d^2-9d+\frac{12}{d})\omega.
\end{equation}
The left-hand side of \eqref{eq:positivityinequality} is bounded from below by using $(d+1)\omega\ge 3$ and $\omega<1$. The right-hand side is bounded from above by the easy-to-prove observation that the last term is at most $9$ for $d\in[3,6]$, and that $9d^2-40d+36\le 9(d-2)^2$. It thus suffices to prove that for $d\in [3,6]$ and for all $\tau\ge 1$,
\begin{equation}\label{eq:positivityinequality2}
\phi_4(\tau,d):=6(d\tau^3-\tau)(d(d-3)\tau^2+3) -  9\sqrt{d}(d-2)^2\tau^2 - 9 \ge 0.
\end{equation}
The obvious bounds
\[
d \tau^3 - \tau \ge (d-1) \tau^3, \qquad
d(d-3) \tau^2 + 3 \ge d^2 - 3d + 3, \qquad
\sqrt{d} \le 3,
\]
imply
\[
\begin{split}
\phi_4(\tau,d) &\ge 6 (d-1) \tau^3 (d^2 - 3d + 3) - 27 (d-2)^2 \tau^2  - 9\\ 
&=\tau^3\left(6 (d-1)(d^2 - 3d + 3) - \frac{27 (d-2)^2}{\tau}  - \frac{9}{\tau^3}\right)\\
&\ge \tau^3\left(6 (d-1)(d^2 - 3d + 3) -27 (d-2)^2  - 9\right) = 3\tau^3(d-3)^2(2d-5)\ge 0,
\end{split}
\]
proving \eqref{eq:positivityinequality2}.

Finally, by the two monotonicity properties just established,
\[
\Psi(\tau, d)\ge\Psi(1,3)
=
1-\frac{1}{\sqrt3 (3-\omega)}+\frac{1}{12 (3-\omega)^2}
>
1-\frac{1}{2\sqrt3}
>0.
\]
Thus $\Psi(\tau, d)$ is positive and strictly increasing in both $\tau$ and $d$, while the prefactor $\frac12\sqrt{\frac\pi2}\,\er^{-\omega\tau^{-2}}$ is positive, independent of $d$, and strictly increasing in $\tau$. Therefore $\mathcal{T}_{d,3}(\tau)$ is monotone increasing both in $\tau\ge1$ and in $d\ge3$, which completes the proof of Proposition \ref{prop:Q3b}.

\subsection{Proof of Lemma \ref{lem:constantmonotone}}\label{proof:constant}

As $\mathfrak{F}_d(\lambda) = \mathfrak{Z}_d(\mathfrak{X}(\lambda))$ and $\mathfrak X(\lambda)$ is strictly increasing in $\lambda$, it is enough to show that
\[
\frac{\partial \log  \mathfrak{Z}_d(\mathfrak{X})}{\partial \mathfrak{X}}
=
\frac1{\mathfrak X+\frac{n-1}{2}}
+\sum_{k=1}^{n-1}\frac1{\mathfrak X+k}
-\frac n{2\mathfrak X}
-\frac{n+2}{2(\mathfrak X+n)}<0,
\]
where we denoted $n:=d-2\ge 1$.

For $1\le k\le n-1$,
\[
\frac1{\mathfrak X+k}
+\frac1{\mathfrak X+n-k} = \frac{2\mathfrak X + n}{\mathfrak X^2+n\mathfrak X+k(n-k)} \le \frac{2\mathfrak X + n}{\mathfrak X^2+n\mathfrak X}
=
\frac1{\mathfrak X}
+\frac1{\mathfrak X+n}.
\]
Therefore
\[
\sum_{k=1}^{n-1}\frac1{\mathfrak X+k}
\le
\frac{n-1}{2}
\left(
\frac1{\mathfrak X}
+\frac1{\mathfrak X+n}
\right).
\]
Consequently
\[
\frac{\partial \log  \mathfrak{Z}_d(\mathfrak{X})}{\partial \mathfrak{X}}
\le
\frac1{\mathfrak X+\frac{n-1}{2}}
-\frac1{2\mathfrak X}
-\frac3{2(\mathfrak X+n)}=
-\frac{
4\mathfrak X^2+2(n-2)\mathfrak X+n(n-1)
}{
2\mathfrak X(\mathfrak X+n)(2\mathfrak X+n-1)
}<0
\]
for all $\mathfrak{X}>0$ and all $n\ge 2$.

\subsection{Proof of Lemma \ref{lem:piecewise-unimodal}}\label{proof:testNF}

We have 
\[
\frac{\partial \log  \mathfrak{Z}_d(\mathfrak{X})}{\partial \mathfrak{X}}
=\mathfrak{G}_d(\mathfrak{X})-\frac{d}{2}\cdot\frac{2\mathfrak{X}+d-2}{\mathfrak{H}_d(\mathfrak{X})},
\]
where 
\[
\mathfrak{G}_d(\mathfrak{X}):=\frac{1}{\mathfrak{X}+\frac{d-3}{2}}+\sum_{j=0}^{d-3}\frac{1}{\mathfrak{X}+j}=\sum_{t\in Z_d} \frac{1}{\mathfrak{X}+t}, \qquad \mathfrak{H}_d(\mathfrak{X}):=\mathfrak{X}(\mathfrak{X}+d-2)+\mathfrak{b}_d,
\]
and $Z_d=\{0,\dots,d-3\}\cup\left\{\frac{d-3}{2}\right\}$ is interpreted as multi-set with $d-1$ elements. 
Multiplying by the positive quantity \(2\mathfrak{H}_d(\mathfrak{X})\), we see that the
critical points of $\mathfrak{Z}_d$ are the zeros of
\begin{equation}\label{eq:Jfrak}
\mathfrak{J}_d(\mathfrak{X}):= 2\mathfrak{H}_d(\mathfrak{X})\mathfrak{G}_d(\mathfrak{X})-d(2\mathfrak{X}+d-2).
\end{equation}
Clearly
\[
\mathfrak{J}_d(\mathfrak{X}) = 
\begin{cases} 
\frac{4\mathfrak{b}_d}{\mathfrak{X}}+O(1),\quad&\text{if }d=3\\
\frac{2\mathfrak{b}_d}{\mathfrak{X}}+O(1),\quad&\text{if }d\ge 4
\end{cases}
\to+\infty \qquad\text{as }\mathfrak{X}\to 0^+,
\]
and 
\[
\mathfrak{J}_d(\mathfrak{X})=-2\mathfrak{X}+O(1)\to-\infty\qquad\text{as }\mathfrak{X}\to+\infty.
\]
Thus, $\mathfrak{J}_d(\mathfrak{X})$ changes sign at least once on $(0, +\infty)$, and it remains to show that it is strictly decreasing.

We first note a lower bound
\begin{equation}\label{eq:bbound}
\mathfrak{b}_d=\frac{\mathcal{I}_{d,1}}{\mathcal{I}_{d,2}}\ge \frac{(d-2)^2}{4}.
\end{equation}
Using the explicit formulae for the integrals, a direct calculation gives
\begin{equation}\label{eq:Idiff}
\mathcal{I}_{d,1}-\frac{(d-2)^2}{4}\mathcal{I}_{d,2}
=
\frac{6 d^2-16 d+24+3 d (d-2)^2 \log \left(1-\frac{2}{d}\right)}{4 d^3},
\end{equation}
which after substitution of the bound \eqref{eq:loglower} becomes $\mathcal{I}_{d,1}-\frac{(d-2)^2}{4}\mathcal{I}_{d,2}\ge \frac{d+6}{2d^3}>0$ as claimed. Therefore, we also have
\begin{equation}\label{eq:Hfrakbound}
\mathfrak{H}_d(\mathfrak{X})\ge \left(\mathfrak{X}+\frac{d}{2}-1\right)^2.
\end{equation}

As
\[
\mathfrak{H}'_d(\mathfrak{X})=2\mathfrak{X}+d-2,
\qquad
\mathfrak{G}'_d(\mathfrak{X})=-\sum_{t\in Z_d}\frac{1}{(\mathfrak{X}+t)^2}<0,
\]
we obtain, using \eqref{eq:Jfrak},  \eqref{eq:Hfrakbound}, and the fact that $d=1+\sum_{t\in Z_d} 1$, 
\[
\begin{split}
\mathfrak{J}'_d(\mathfrak{X}) &= 2(2\mathfrak{X}+d-2)\mathfrak{G}_d(\mathfrak{X})+2\mathfrak{H}_d(\mathfrak{X})\mathfrak{G}'_d(\mathfrak{X}) - 2d\\
&\le 2\left((2\mathfrak{X}+d-2)\mathfrak{G}_d(\mathfrak{X})+\left(\mathfrak{X}+\frac{d}{2}-1\right)^2\mathfrak{G}'_d(\mathfrak{X}) -d\right)\\
&=-2+2\sum_{t\in Z_d}\left(\frac{(2\mathfrak{X}+d-2)(\mathfrak{X}+t)-\left(\mathfrak{X}+\frac{d}{2}-1\right)^2}{(\mathfrak{X}+t)^2}-1\right)\\
&=-2-\sum_{t\in Z_d}\frac{(d-2 (t+1))^2}{2(\mathfrak{X}+t)^2} \le -2 <0.
\end{split}
\]
Hence $\mathfrak{J}_d(\mathfrak{X})$ is strictly decreasing,  from $+\infty$ to $-\infty$,   on the interval $(0,+\infty)$, and therefore has a single positive zero. Thus, $\mathfrak{Z}_d(\mathfrak{X})$ has a single critical point, a strict maximum, on the interval  $(0,+\infty)$, and as $\mathfrak{X}(\lambda)$ is monotone increasing, the same is true for $\mathfrak{F}_d(\lambda)$. 

\subsection{Proof of Proposition \ref{prop:othertestvalues}}\label{proof:prop:othertestvalues}
Set 
\[
\mathfrak{Q}=\mathfrak{Q}_d(\lambda):=\frac{\mathfrak{Y}(\lambda)}{\lambda}
\]
and
\begin{equation}\label{eq:PhiI}
\mathfrak{q}_d:=\mathcal{I}_{d,1}-\left(\frac{d}{2}-1\right)^2\mathcal{I}_{d,2}.
\end{equation}
A comparison of \eqref{eq:Y2defn}, \eqref{eq:ab}, and \eqref{eq:PhiI} using the explicit expression for $\mathcal{I}_{d,3}$ from \eqref{eq:INF} yields
\begin{equation}\label{eq:q2id}
\mathfrak{Q}^2=\mathfrak{a}_d\left(1-\frac{d^4\mathfrak{q}_d}{4\lambda^2}\right).
\end{equation}

\begin{lemma}\label{lem:q2}
With the choice of the test-function \eqref{eq:rhoNF}, for every $d\ge 4$, with $\epsilon=2/d$,
\begin{equation}\label{eq:series}
\frac{1}{\mathfrak{a}_d}=\sum_{n=0}^{\infty}\frac{24\epsilon^{n}}{(n+2)(n+3)(n+4)}=1+\frac{2\epsilon}{5}+\dots,
\end{equation}
\begin{equation}\label{eq:qseries}
\mathfrak{q}_d=\frac{\epsilon^2}{8}+\frac{\epsilon^3}{2}-\frac{3}{2}\sum_{n=4}^{\infty}\frac{\epsilon^{n}}{n(n-1)(n-2)}.
\end{equation}
Moreover,
\begin{equation}\label{eq:sw}
\frac{1}{1+\mathfrak{s}_d}\le \mathfrak{a}_d< 1, \qquad \text{with }\quad\mathfrak{s}_d:=\frac{4(d-1)}{5d(d-2)}
\end{equation}
\begin{equation}\label{eq:sw1}
0\le \frac{d}{4}\mathfrak{q}_d\le\mathfrak{t}_d, \qquad \text{with }\quad\mathfrak{t}_d:=\frac{1}{8d}+\frac{1}{d^{2}}.
\end{equation}
\end{lemma}

\begin{proof}
We have, by \eqref{eq:INF} and \eqref{eq:epsL},
\[
\frac{1}{\mathfrak{a}_d}=\frac{\mathcal{I}_{d,2}}{\mathcal{I}_{d,3}}=\frac{d^4}{4}\,\mathcal{I}_{d,2}
=\frac{d}{4}\left(3d(d-2)^2 \ell_d-6d^2+18d-8\right)
=\frac{12(1-\epsilon)^2}{\epsilon^4} \ell_d-\frac{12}{\epsilon^3}+\frac{18}{\epsilon^2}-\frac{4}{\epsilon}.
\]
Substituting the series expansion 
\[
\ell_d=-\log(1-\epsilon)=\sum_{k\ge1}\frac{\epsilon^{k}}{k},
\]
and collecting the terms, we arrive at \eqref{eq:series}. Clearly,  $\frac{1}{\mathfrak{a}_d}>1$. As for $n\ge2$ one has $(n+2)(n+3)(n+4)\ge 120$, we get
\[
\frac{1}{\mathfrak{a}_d}-1\le \frac{2\epsilon}{5}+\frac15\sum_{n\ge2}\epsilon^{n}
=\frac{2\epsilon}{5}+\frac{\epsilon^2}{5(1-\epsilon)}
=\frac{\epsilon(2-\epsilon)}{5(1-\epsilon)}
=\frac{4(d-1)}{5d(d-2)}=\mathfrak{s}_d,
\]
which is equivalent to \eqref{eq:sw}.

By \eqref{eq:PhiI} and \eqref{eq:Idiff},
\[
\mathfrak{q}_d=\frac{6d^2-16d+24-3d(d-2)^2 \ell_d}{4d^3}
=\frac{3\epsilon}{4}-\epsilon^2+\frac{3\epsilon^3}{4}-\frac{3}{4}(1-\epsilon)^2\ell_d.
\]
Expanding $\ell_d$ as above and collecting the terms proves \eqref{eq:qseries}, and gives the upper bound in \eqref{eq:sw1}. Also, $\mathfrak{q}_d\ge0$ due to \eqref{eq:bbound}.
\end{proof}

We now take 
\[
\tau^\bullet\in \left\{\tau_-^\bullet,\tau_+^\bullet\right\} = \left\{\frac{4}{5},\frac{28}{25}\right\}, 
\qquad \lambda^\bullet:={\tau^\bullet}^3 d^{3/2},
\qquad\mathfrak{Y}^\bullet:=\mathfrak{Y}\left(\lambda^\bullet\right),
\qquad\mathfrak{Q}^\bullet:=\frac{\mathfrak{Y}^\bullet}{\lambda^\bullet},
\]
and use \eqref{eq:N1stbound} to get 
\begin{equation}\label{eq:N1stboundbullet}
\begin{split}
\mathfrak{F}_d\left({\tau^\bullet}^3 d^{3/2}\right) &\ge {\tau^\bullet}^{-3} \mathfrak{c}_d \left(\frac{\mathfrak{Y}^\bullet}{\lambda^\bullet}\right)^{d-1} \left(1-\frac{d-1}{2\mathfrak{Y}^\bullet}-\frac{(d-1)(d-2)(d-6)}{24{\mathfrak{Y}^\bullet}^2}\right)\\
&= {\tau^\bullet}^{-3} \mathfrak{c}_d {\mathfrak{Q}^\bullet}^{d-1} \left(1-\frac{d-1}{2\lambda^\bullet\mathfrak{Q}^\bullet}-\frac{(d-1)(d-2)(d-6)}{24{\lambda^\bullet}^2{\mathfrak{Q}^\bullet}^2}\right).
\end{split}
\end{equation}
We now deduce a lower bound of the right-hand side of \eqref{eq:N1stboundbullet}, acting term-by-term.

By \eqref{eq:q2id} and Lemma \ref{lem:q2}, since ${\tau^\bullet}^{-6}\mathfrak{t}_d \le \left(\frac{5}{4}\right)^6 \left(\frac{1}{32}+\frac{1}{16}\right)<1$ for $d\ge 4$, 
\begin{equation}\label{eq:q2lb}
{\mathfrak{Q}^\bullet}^2=\mathfrak{a}_d\left(1-\frac{d \mathfrak{q}_d}{4{\tau^\bullet}^6}\right)
\ge \mathfrak{a}_d\left(1-{\tau^\bullet}^{-6}\mathfrak{t}_d\right)
\ge \frac{1-{\tau^\bullet}^{-6}\mathfrak{t}_d}{1+\mathfrak{s}_d}
=\frac{5d(d-2)}{5d^2-6d-4}-\frac{5 (d-2) (d+8)}{8 d \left(5 d^2-6 d-4\right){\tau^\bullet}^6}
=:{\utilde{\mathfrak{Q}}^\bullet}^2.
\end{equation}
As both $\mathfrak{s}_d$ and $\mathfrak{t}_d$ are decreasing in $d$, the right-hand side of \eqref{eq:q2lb} increases in $d$. From \eqref{eq:q2lb}, 
\[
\log {\mathfrak{Q}^\bullet}^{d-1}\ge\frac{d-1}{2}\left(\log\left(1-{\tau^\bullet}^{-6}\mathfrak{t}_d\right)-\log\left(1+\mathfrak{s}_d)\right)\right)
\ge \frac{d-1}{2}\left(-\frac{\mathfrak{t}_d}{{\tau^\bullet}^{6}-\mathfrak{t}_d}-\mathfrak{s}_d\right),
\]
where we have used $\log(1+\mathfrak{s}_d)\le \mathfrak{s}_d$ and, by Lemma \ref{lem:log} with $x={\tau^\bullet}^{-6}\mathfrak{t}_d$,
\[
\log\left(1-{\tau^\bullet}^{-6}\mathfrak{t}_d\right)>-{\tau^\bullet}^{-6}\mathfrak{t}_d-\frac{{\tau^\bullet}^{-12}\mathfrak{t}_d^2}{2(1-{\tau^\bullet}^{-6}\mathfrak{t}_d)}>-\frac{{\tau^\bullet}^{-6}\mathfrak{t}_d}{1-{\tau^\bullet}^{-6}\mathfrak{t}_d}=-\frac{\mathfrak{t}_d}{{\tau^\bullet}^{6}-\mathfrak{t}_d}.
\]
Hence, substituting the definitions of $\mathfrak{s}_d$ and $\mathfrak{t}_d$, we conclude that 
\begin{equation}\label{eq:q2lb1}
{\mathfrak{Q}^\bullet}^{d-1} \ge \exp\left(-\frac{2(d-1)^2}{5d(d-2)}\right)\,\exp\left(-\frac{(d-1) (d+8)}{2 \left(8 d^2 {\tau^\bullet}^{6}-d-8\right)}\right),
\end{equation}
and it is easily verified that both factors in the right-hand side of \eqref{eq:q2lb1} are increasing in $d$ for $d\ge 4$.

Finally, acting as in the proof of Proposition \ref{prop:constant2}, we estimate
\begin{equation}\label{eq:Kbound}
1-\frac{d-1}{2\lambda^\bullet\mathfrak{Q}^\bullet}-\frac{(d-1)(d-2)(d-6)}{24{\lambda^\bullet}^2{\mathfrak{Q}^\bullet}^2}\ge 
1-\frac{1}{2{\tau^\bullet}^3\sqrt{d} \utilde{\mathfrak{Q}}^\bullet}-\frac{1}{24{\tau^\bullet}^{6}{\utilde{\mathfrak{Q}}^\bullet}^2},
\end{equation}
with the right-hand side also monotone increasing in $d$.

Combining \eqref{eq:N1stboundbullet}--\eqref{eq:Kbound} and  \eqref{eq:cflower}, we obtain the lower bound
\begin{equation}\label{eq:bulletlower}
\mathfrak{F}_d\left({\tau^\bullet}^3 d^{3/2}\right) \ge \frac{193}{77{\tau^\bullet}^3} \exp\left(-\frac{2(d-1)^2}{5d(d-2)}-\frac{(d-1) (d+8)}{2 \left(8 d^2 {\tau^\bullet}^{6}-d-8\right)}\right)
\left(1-\frac{1}{2{\tau^\bullet}^3\sqrt{d} \utilde{\mathfrak{Q}}^\bullet}-\frac{1}{24{\tau^\bullet}^{6}{\utilde{\mathfrak{Q}}^\bullet}^2}\right),
\end{equation}
valid for $\tau^\bullet=\tau^\bullet_\pm$, which is monotone increasing in $d$. Evaluating the right-hand side of \eqref{eq:bulletlower} at $\tau^\bullet=\tau^\bullet_-$, $d=11$, and at $\tau^\bullet=\tau^\bullet_+$, $d=16$, gives  
\[
\mathfrak{F}_{11}\left({\tau^\bullet_-}^3 11^{3/2}\right)\gtrsim 1.08168 >1, \qquad  \mathfrak{F}_{16}\left({\tau^\bullet_+}^3 16^{3/2}\right)\gtrsim 1.011378 > 1,
\]
thus ensuring that $\mathfrak{F}_d\left({\tau^\bullet_-}^3 d^{3/2}\right) >1$ for all $d\ge 11$, and $\mathfrak{F}_d\left({\tau^\bullet_+}^3 d^{3/2}\right) >1$ for all $d\ge 16$. 

In the remaining range of dimensions $d$,  we evaluate $\mathfrak{F}_d\left({\tau^\bullet_\pm}^3 d^{3/2}\right)$ directly using \eqref{eq:Fd}, obtaining
\[
\begin{array}{ccc@{\qquad\qquad}ccc}
d & \mathfrak{F}_d\left({\tau_-^\bullet}^{3}d^{3/2}\right) & \mathfrak{F}_d\left({\tau_+^\bullet}^{3}d^{3/2}\right)
& d & \mathfrak{F}_d\left({\tau_-^\bullet}^{3}d^{3/2}\right) & \mathfrak{F}_d\left({\tau_+^\bullet}^{3}d^{3/2}\right)\\\hline
 4 & 1.0728 & 1.0510 & 10 & 1.5183 & 1.0511\\
 5 & 1.2166 & 1.0468 & 11 & & 1.0528\\
 6 & 1.3114 & 1.0460 & 12 &  & 1.0545\\
 7 & 1.3813 & 1.0466 & 13 & & 1.0562\\
 8 & 1.4361 & 1.0478 & 14 &  & 1.0577\\
 9 & 1.4808 & 1.0494 & 15 &  & 1.0592
\end{array}
\]
All of these values exceed one, which completes the proof. \qed

\section{Rational enclosures for zeros of derivatives of ultraspherical Bessel functions}\label{app:rational}

The algorithm of \S\ref{sec:gap} is built around the following routines.  

\

\noindent\textbf{Procedure} $\mathtt{Sign}(X,x)$.

\noindent\textit{Input}: a function $X\in\left\{F_\nu, H_{d,m}\right\}$, $x\in\mathbb{Q}_{>0}$. 

\noindent\textit{Output}: in $\{-1,1\}$.

\noindent\textit{Workflow}:
\begin{enumerate} 
\item Set $t=\frac{x^2}{4}$, $N=8$. 
\item Compute the rational enclosure for $X(t)$ by Lemma \ref{lem:enclosures} with this $N$.
\item If the enclosure does not contain zero, return the common sign of its endpoints; otherwise, set $N:=2N$ and repeat step 2.
\end{enumerate}

\noindent\textit{Certification}: by Lemma \ref{lem:enclosures}.

\noindent\textit{Termination}:  since, by Lemma \ref{lem:trans}, all zeros of either $F_\nu$, or $H_{d,m}$ are irrational, the true value $X(x)\ne 0$.

\

\noindent\textbf{Procedure} $\mathtt{BesselZeroBrackets}(\nu,\Lambda)$.

\noindent\textit{Input}: $\nu\in\frac{1}{2}\mathbb{Z}$, $\Lambda\in\mathbb{Q}_{>0}$.

\noindent\textit{Output}: $K\in \mathbb{N}$ and rational numbers $0<a_1<b_1<a_2<\dots<a_K<b_K$ such that $j_{\nu,k}\in (a_k, b_k)$, $b_k-a_k=3$, and $a_K\ge \Lambda$.

\noindent\textit{Workflow}:
\begin{enumerate} 
\item Set $x_0=\nu$ and $s_0=+1$.
\item For $i\in\mathbb{N}$, set $x_i = x_{i-1}+3$, and compute $s_i = \mathtt{Sign}(F_\nu,x_i)$.
\item Whenever $s_{i-1} s_i = -1$, record the bracket $(x_{i-1},x_i)$, labelling them  $(a_k,b_k)$ in order of appearance, $k=1,2,\dots$.
\item Stop at the first recorded bracket with $a_k\ge \Lambda$  and set $K=k$. 
\end{enumerate}

\noindent\textit{Certification}: as easily follows from \cite{LorSze64}, we have $j_{\nu,k}-j_{\nu,k-1}\ge \pi>3$ for all $k\in\mathbb{N}$ and all $\nu\ge\frac{1}{2}$, therefore each interval of length $3$ may contain at most one zero (which is simple). As $j_{\nu,1}>\nu$,  this proves that the recorded brackets enumerate, in increasing order and without omission, the zeros $j_{\nu,1},j_{\nu,2}, \dots$.

\noindent\textit{Termination}:  As there is a finite number of zeros in each interval $(0,\Lambda]$, the process will terminate.  

\

\noindent\textbf{Procedure} $\mathtt{Refine}(X,(a,b))$.

\noindent\textit{Input}: a function $X\in\left\{F_\nu, H_{d,m}\right\}$, and a rational bracket $(a,b)$, $a<b$ with $\mathtt{Sign}(X,a)\ne \mathtt{Sign}(X,b)$ (exception: at $a=0$, the value $H_{d,m}(0)=m$ is used instead of  $\mathtt{Sign}(H_{d,m},0)$).

\noindent\textit{Output}: a bracket of half the width with the same property, containing the same unique zero: evaluate $\mathtt{Sign}\left(X,\frac{a+b}{2}\right)$, and keep the half with different signs at the ends.

\

\noindent\textbf{Procedure} $\mathtt{MAIN}(d,\Lambda,\epsilon)$.

\noindent\textit{Input}: an integer $d\ge 3$, a rational $\Lambda>0$ and a rational tolerance $\epsilon>0$. 

\noindent\textit{Output}: the \emph{complete} list $\mathtt{ListP}$ as described in \S\ref{sec:gap}.

We mention that by \eqref{eq:LS}, we only need to examine the finite number of $m$'s such that 
\[
0\le m\le \mathtt{M}_{d,\Lambda}=\max\{m\ge 0: m(m+d-2)<\Lambda^2\}.
\]

\noindent\textit{Workflow}:
\begin{enumerate} 
\item[0.] Initialise $\mathtt{ListP}:=\{(0,1,0)\}$ (the exceptional zero $p'_{d,0,1}=0\le\Lambda$). Compute $\mathtt{M}_{d,\Lambda}$: this is a finite integer computation.
\item[1.] \emph{The case $m=0$.} Run $\mathtt{BesselZeroBrackets}(d/2,\Lambda)$, giving brackets $(a_k,b_k)\ni j_{d/2,k}$, $k\le K$. For each $k$, repeatedly $\mathtt{Refine}$ the bracket for $F_{d/2}$ until either 
\begin{enumerate}
\item[(i)] $b_k\le\Lambda$ and $b_k-a_k\le\epsilon$: append $\left(0, k+1,\beta:=b_k\right)$ to $\mathtt{ListP}$ (recall $p'_{d,0,k+1}=j_{d/2,k}$), or
\item[(ii)] $a_k\ge\Lambda$: discard $(0,k+1)$ and all larger indices, and finish the case $m=0$. By construction $a_K\geq \Lambda$, so case (ii) occurs at the latest at $k=K$; all zeros with $k>K$ satisfy $j_{d/2,k}>j_{d/2,K}>a_K\ge\Lambda$.
\end{enumerate}
\item[2.] \emph{The cases $1\le m\le \mathtt{M}_{d,\Lambda}$.} For each such $m$ set $\nu=m+d/2-1$ and
\begin{enumerate}
\item[2.1.] Run $\mathtt{BesselZeroBrackets}(\nu,\Lambda)$, giving $K\ge1$ and brackets $(a_k,b_k)\ni j_{\nu,k}$. $\mathtt{Refine}$ each bracket (for $F_\nu$) until the \emph{ultraspherical derivative sign conditions}
\begin{equation}\label{eq:signconditions}
\mathtt{Sign}(H_{d,m},a_k)=\mathtt{Sign}(H_{d,m},b_k)=(-1)^k,\qquad k=1,\dots,K,
\end{equation}
hold. (\textit{Step termination:} as $a_k\uparrow j_{\nu,k}$ and $b_k\downarrow j_{\nu,k}$, both signs eventually equal $(-1)^k$ by Lemma \ref{lem:pandj}. Each $\mathtt{Refine}$ halves the width. Note that refining preserves $a_K\ge\Lambda$, since $a_K$ can only increase.)
\item[2.2.] Form the \emph{ultraspherical derivative  windows}
\[
\mathtt{W}_0:=(0, a_1),\qquad \mathtt{W}_k:=(b_k,a_{k+1}),\quad k=1,\dots,K-1 .
\]
By \eqref{eq:signconditions}, $H_{d,m}$ has certified opposite signs at the endpoints of each $\mathtt{W}_k$ (at the left endpoint of $\mathtt{W}_0$ the exact value $H_{d,m}(0)=m>0$ is used). By Lemma \ref{lem:pandj} and \eqref{eq:signconditions}, $\mathtt{W}_k$ contains \emph{exactly one} zero of $H_{d,m}$, namely $p'_{d,m,k+1}$; and no zero of $H_{d,m}$ in $(0,j_{\nu,K})$ lies outside $\bigcup_k \overline{\mathtt{W}_k}$, while $p'_{d,m,K+1}>j_{\nu,K}>a_K\ge\Lambda$.
\item[2.3.] For each $k=0,\dots,K-1$, repeatedly $\mathtt{Refine}$ the window $\mathtt{W}_k=(w_l,w_r)$ for $H_{d,m}$ until either
\begin{enumerate}
\item[(i)] $w_r\le\Lambda$ and $w_r-w_l\le\epsilon$: append $\left(m, k+1, w_r\right)$ to $\mathtt{ListP}$, or
\item[(ii)] $w_l\ge\Lambda$: discard $(m,k+1)$.
\end{enumerate}
(\textit{Step termination:}  let $p':=p'_{d,m,k+1}$; the bracket endpoints satisfy $w_l<p'<w_r$ with $w_r-w_l\to0$ under bisection, and $p'\ne\Lambda$ by Lemma \ref{lem:trans}; if $p'<\Lambda$ then eventually both $w_r\le\Lambda$ and $w_r-w_l\le\epsilon$ hold, while if $p'>\Lambda$ then eventually $w_l\ge\Lambda$.)
\end{enumerate}
\item[3.] Return $\mathtt{ListP}$.
\end{enumerate}

\end{appendices}

\end{document}